\documentclass[12pt,reqno]{article}
\usepackage[usenames]{color}
\usepackage{amssymb}
\usepackage{amsmath}
\usepackage{amsthm}
\usepackage{amsfonts}
\usepackage{amscd}
\usepackage{graphicx}
\usepackage[usenames,dvipsnames]{xcolor}
\usepackage{color}
\usepackage{graphics}
\usepackage{latexsym}
\usepackage{epsf}

\usepackage[colorlinks=true,
linkcolor=webgreen,
filecolor=webbrown,
citecolor=webgreen]{hyperref}
\usepackage{hyphenat}
\definecolor{webgreen}{rgb}{0,.5,0}
\definecolor{webbrown}{rgb}{.6,0,0}

\usepackage{fullpage}
\usepackage{float}

\usepackage{graphics}
\usepackage{latexsym}
\usepackage{epsf}

\newcommand{\seqnum}[1]{\href{https://oeis.org/#1}{\rm \underline{#1}}}

\usepackage{tikz}
\usepackage{mathrsfs}

\begin{document}
	\theoremstyle{plain}
	\newtheorem{theorem}{Theorem}
	\newtheorem{corollary}[theorem]{Corollary}
	\newtheorem{lemma}[theorem]{Lemma}
	\newtheorem{proposition}[theorem]{Proposition}
	\theoremstyle{definition}
	\newtheorem{definition}[theorem]{Definition}
	\newtheorem{notation}[theorem]{Notation}
	\newtheorem{example}[theorem]{Example}
	\newtheorem{conjecture}[theorem]{Conjecture}
	\theoremstyle{remark}
	\newtheorem{remark}[theorem]{Remark}
	
		\begin{center}
		\vskip 1cm{\LARGE\bf 
		 On The Emergence of a New Prime Number And Omega Sequences}
		\vskip 1cm
		\large
		Moustafa Ibrahim\\
		Department of Mathematics\\
		College of Science\\
		University of Bahrain\\
		 Kingdom of Bahrain\\
		\href{mailto:mimohamed@uob.edu.bh}{\tt mimohamed@uob.edu.bh}
	\end{center}
	
	\vskip .2 in
	
	\begin{abstract}
		This paper highlights the emergence of the Omega sequence in number theory and its connection with the emergence of a new prime number, and also highlights its  theoretical applications for Lucas-Lehmer primality test, and Euclid-Euler theory for even perfect numbers. We also show that Omega sequences unify and give new representations for Mersenne numbers, Fermat numbers, Lucas numbers, Fibonacci numbers, Chebyshev sequence, Dickson sequence, and others.  
	\end{abstract}
	
		\section{Summary for the main results}
	For a natural number $n$, we define $\delta(n)=n\pmod{2}$. For an arbitrary real number $x$,  $\lfloor{\frac{x}{2}}\rfloor$ is the highest integer less than or equal $\frac{x}{2}$. Due to the work of Euclid and Euler, it is well-known that an even integer $n$ is perfect if and only if $n = 2^{p-1}(2^p-1)$, where $2^{p}-1$ is prime. A prime of the form $2^p-1$ is called a \textit{Mersenne} prime. Here is a summary of some results.

	\begin{theorem}{(Lucas-Lehmer-Moustafa)}
		\label{Theorem of ABC1}
		For any given prime $p\geq 5$, $n:=2^{p-1}$, we associate the double-indexed polynomial sequences $A_r(k)$, $B_r(k)$, which are defined by
				\begin{equation}
			\label{ABC2} 
			\begin{aligned}		
				A_r(k) &=  (p-r-k) \: A_r(k-1) + \: 4 \: (p-2r) \: A_{r+1}(k-1), \quad &A_r(0) = 1  \quad   \text{for all} \:\: r,\\
		B_r(k) &=  -2 \: (n-r-k) \: B_r(k-1) - \:2 \: (n-2r-1) \: B_{r+1}(k-1),  \quad &B_r(0) = 1  \quad   \text{for all} \:\: r.
			\end{aligned}
		\end{equation}
		Then both of the ratios   
		\begin{equation}
			\label{ABC3} 
			\begin{aligned}
				\frac{A_0( \lfloor{\frac{p}{2}}\rfloor   )}{(p-1)(p-2) \cdots (p - \lfloor{\frac{p}{2}}\rfloor )} \quad	, \quad 	\frac{B_0( \lfloor{\frac{n}{2}}\rfloor   )}{(n-1)(n-2) \cdots (n - \lfloor{\frac{n}{2}}\rfloor )} 	
			\end{aligned}
		\end{equation}
		are integers. Moreover the number $2^p -1$ is prime \bf{if and only if} 
		\begin{equation}
			\label{ABC4} 
			\begin{aligned}
				\frac{A_0( \lfloor{\frac{p}{2}}\rfloor   )}{(p-1)(p-2) \cdots (p - \lfloor{\frac{p}{2}}\rfloor )}  \quad  \vert \quad \frac{B_0( \lfloor{\frac{n}{2}}\rfloor   )}{(n-1)(n-2) \cdots (n - \lfloor{\frac{n}{2}}\rfloor )}.
			\end{aligned}
		\end{equation}
		
	\end{theorem}
	
		\begin{theorem}{(New representation of Mersenne numbers)}
		\label{Theorem of G2fQ}
		For any given odd natural number $p$, the number $2^p -1$ can be represented by
		\begin{equation}
			\label{G2BfQ} 
			\begin{aligned}
				2^p-1 \:=\: \frac{A_0( \lfloor{\frac{p}{2}}\rfloor   )}{(p-1)(p-2) \cdots (p - \lfloor{\frac{p}{2}}\rfloor )}, 	
			\end{aligned}
		\end{equation}
		
		where the double-indexed polynomial sequence $A_r(k)$ is defined by the recurrence relation 
		\begin{equation}
			\label{G2-e-q} 
			\begin{aligned}
				A_r(k) &=  (p-r-k) \: A_r(k-1) + \: 4 \: (p-2r) \: A_{r+1}(k-1), \quad &A_r(0) = 1  \quad   \text{for all} \:\: r.
			\end{aligned}
		\end{equation}
	\end{theorem}

	\begin{theorem}{}
		\label{KG11}
		For any given natural number $n$, we associate the double-indexed polynomial sequence $U_r(k)$, which is defined by
		\begin{equation}
			\label{KG22} 
			\begin{aligned}
				U_r(k) &=  \: (n-r-k) \: U_r(k-1) - \:2 \: (n-2r-\delta(n-1)) \: U_{r+1}(k-1), \\
				 \quad U_r(0) &= 1   \quad   \text{for all} \:\: r.
			\end{aligned}
		\end{equation}
		Then 	
		\begin{equation}
			\label{KG44} 
			\frac{U_0( \lfloor{\frac{n}{2}}\rfloor   )}{(n-1)(n-2) \cdots (n - \lfloor{\frac{n}{2}}\rfloor )} 	\:=
			\begin{cases}
				+2	      &   n \equiv \pm 0       \pmod{6} \\
				+1	      &   n \equiv \pm 1       \pmod{6} \\
				-1        &   n \equiv \pm 2       \pmod{6} \\
				-2	      &   n \equiv \pm 3      \pmod{6} 
			\end{cases}.    
		\end{equation}
	\end{theorem}
	
		\begin{theorem}{}
		\label{KG11Q}
		For any given natural number $n$, we associate the double-indexed polynomial sequence $V_r(k)$, which is defined by
		\begin{equation}
			\label{KG22Q} 
			\begin{aligned}
				V_r(k) &=  +2 \: (n-r-k) \: V_r(k-1) - \:2 \: (n-2r-\delta(n-1)) \: V_{r+1}(k-1), \\ \quad V_r(0) &= 1   \quad   \text{for all} \:\: r.
			\end{aligned}
		\end{equation}
		Then 	
		\begin{equation}
			\label{KG44Q} 
			\frac{V_0( \lfloor{\frac{n}{2}}\rfloor   )}{(n-1)(n-2) \cdots (n - \lfloor{\frac{n}{2}}\rfloor )} 	\:=
			\begin{cases}
				+2	      &   n \equiv \pm 0       \pmod{8} \\
				+1	      &   n \equiv \pm 1       \pmod{8} \\
				\: 0	      &   n \equiv \pm 2       \pmod{8} \\				
				-1        &   n \equiv \pm 3       \pmod{8} \\
				-2	      &   n \equiv \pm 4      \pmod{8} 
			\end{cases}.    
		\end{equation}
	\end{theorem}

	\begin{theorem}{}
		\label{PtP}
		For any given natural number $n$, we associate the double-indexed polynomial sequence $W_r(k)$, which is defined by
		\begin{equation}
			\label{GG1-1} 
			\begin{aligned}
				W_r(k) &=  3 \: (n-r-k) \: W_r(k-1) - \:2 \: (n-2r-\delta(n-1)) \: W_{r+1}(k-1), \\ \:\: W_r(0) &= 1   \   \text{for all} \:\: r.
			\end{aligned}
		\end{equation}
		Then 		
		\begin{equation}
			\label{PPP1} 
			\: 	\frac{W_0( \lfloor{\frac{n}{2}}\rfloor   )}{(n-1)(n-2) \cdots (n - \lfloor{\frac{n}{2}}\rfloor )} 	 \:=
			\begin{cases}
				+2	      &   n \equiv \pm 0       \pmod{12} \\
				+1	      &  n \equiv \pm 1 , \pm 2      \pmod{12} \\
				\:\:0        &   n \equiv \pm 3       \pmod{12} \\
				-1        &  n \equiv \pm 4, \pm 5       \pmod{12} \\
				-2	      &   n \equiv \pm 6  \pmod{12}
			\end{cases}.  
		\end{equation}
	\end{theorem}		
	
	\begin{theorem}
		\label{Theorem of G5}
		For any given natural number $n$, we associate the double-indexed polynomial sequence $T_r(k)$,  which is defined by
		\begin{equation}
			\label{G5} 
			\begin{aligned}
				T_r(k) &=  4 \: (n-r-k) \: T_r(k-1)  - 2\: (n-2r - \: \delta(n-1)) \: T_{r+1}(k-1),  \\
				T_r(0) &= 1   \quad   \text{for all} \:\: r.
			\end{aligned}
		\end{equation}
	Then
		\begin{equation}
			\label{G5-2} 
			\begin{aligned}
				 \frac{T_0( \lfloor{\frac{n}{2}}\rfloor   )}{(n-1)(n-2) \cdots (n - \lfloor{\frac{n}{2}}\rfloor )} \:=\:    2^{\delta(n-1)}. 	
			\end{aligned}
		\end{equation}
	\end{theorem}

	\begin{theorem}{(New representation for Lucas sequence)}
		\label{Theorem of G6}
		For any given natural number $n$, we associate the double-indexed polynomial sequence $H_r(k)$, which is defined by
		\begin{equation}
			\label{G6} 
			\begin{aligned}
				H_r(k) &=   \: (n-r-k) \: H_r(k-1) + 2\: (n-2r - \: \delta(n-1)) \: H_{r+1}(k-1), \\
					H_r(0) &= 1   \quad   \text{for all} \:\: r.
			\end{aligned}
		\end{equation}
		Then 
		\begin{equation}
			\label{G6-2} 
			\begin{aligned}
				L(n) \:= \: \frac{H_0(\lfloor{\frac{n}{2}}\rfloor   )}{(n-1)(n-2) \cdots (n - \lfloor{\frac{n}{2}}\rfloor )}, 	
			\end{aligned}
		\end{equation}
	
		where $L(n)$ is Lucas sequence defined by $L(m+1)=L(m) + L(m-1)$, $L(1)=1$, $L(0)=2$.  	\end{theorem}

	\begin{theorem}{(New representation for Fermat numbers)}
		\label{Theorem of G4}
		For any given natural number $n$, we associate the double-indexed polynomial sequence $F_r(k)$, which is defined by
		\begin{equation}
			\label{G4} 
			\begin{aligned}
				F_r(k) &=   \: \big(2^n-r-k\big) \: F_r(k-1) + 4 \: \big(2^n -2r -1 \big) \: F_{r+1}(k-1),  \\
				F_r(0) &= 1   \quad   \text{for all} \:\: r.
			\end{aligned}
		\end{equation}
		Then the Fermat number $F_{n} = 2^{2^n} + 1 $ can be represented by 
		\begin{equation}
			\label{G4-2} 
			\begin{aligned}
				F_{n}  \:=\: \frac{F_0\big(2^{n-1}   \big)}{\big(2^n-1\big)\big(2^n-2\big) \cdots \big(2^{n-1}\big)}. 	
			\end{aligned}
		\end{equation}
	\end{theorem}
	\begin{theorem}{(New representation for Fibonacci-Lucas oscillating sequence)}
		\label{Theorem of G7}
		For any given natural number $n$, we associate the double-indexed polynomial sequence $G_r(k)$, which is defined by
		\begin{equation}
			\label{G7} 
			\begin{aligned}
			G_r(k) &=   5 \: (n-r-k) \: G_r(k-1) - 2\: \big(n-2r-\delta(n-1)\big) \: G_{r+1}(k-1),  \\
				G_r(0) &= 1   \quad   \text{for all} \:\: r.
			\end{aligned}
		\end{equation}
		Then 
		\begin{equation}
			\label{G7-2} 
			\begin{aligned}
				\frac{G_0( \lfloor{\frac{n}{2}}\rfloor   )}{(n-1)(n-2) \cdots (n - \lfloor{\frac{n}{2}}\rfloor )} = \begin{cases}
					F(n)	      &   \text{if} \quad n \:\: odd \\
					L(n)	      &  \text{if} \quad n \: \: even  
				\end{cases},	
			\end{aligned}
		\end{equation}
		where $F(n)$ is Fibonacci sequence defined by $F(m+1)=F(m) + F(m-1), F(1)=1, F(0)=0$, and
	$L(n)$ is Lucas sequence defined by $L(m+1)=L(m) + L(m-1)$, $L(1)=1$, $L(0)=2$.	
		 
	\end{theorem}

	\section{The Omega sequence}

	\begin{definition}{(The Omega sequence associated with $n$ and a point $(\zeta, \xi )$)   }
		\label{omegaDef} \\
		For any given natural number $n$, and a point $(\zeta, \xi)$, $(\zeta, \xi)\neq (0,0)$, we associate the double-indexed sequence $\Omega_r(k|\zeta, \xi|n)$, where $0 \leq r + k \leq  \lfloor{\frac{n}{2}}\rfloor $, such that 
		\begin{equation}
			\label{G0} 
			\begin{aligned}
	\Omega_r\big(k|\zeta, \xi|n\big) &= (2\zeta-\xi)  \: \big(n-r-k\big) \: \Omega_r\big(k-1|\zeta, \xi|n\big) \\
	& \quad - 2\ \zeta  \: \big(n-2r-\delta(n-1)\big) \: \Omega_{r+1}\big(k-1|\zeta, \xi|n\big),  \\
				\Omega_r\big(0|\zeta, \xi|n\big) &= 1   \quad   \text{for all} \:\: r.
			\end{aligned}
		\end{equation}
	\end{definition}

	In Section~\ref{SS}, we give detailed example to compute the Omega sequence associated with some point. 
	
	\begin{notation}
		For a given point $(\zeta, \xi),$ we put \[   \Omega_r\big(k|\zeta, \xi|n\big) = \Omega_r(k|n). \]
		For a given $n$, and point $(\zeta, \xi),$ we put \[   \Omega_r\big(k|\zeta, \xi|n\big) = \Omega_r(k). \]
		
	\end{notation}
	
		\section{ An illustrative example of the calculations of \texorpdfstring{$\Omega$}{}}
	\label{SS} 
	\subsection*{The Omega sequence associated with $(1,-2)$ }
	We know that Omega sequence associated with $n$ and $(1,-2)$ is given by the recurrence relation 
	\begin{equation}
		\label{AU1} 
		\begin{aligned}
			\Omega_r\big(k|1, -2|n\big) &= 4  \: \big(n-r-k\big) \: \Omega_r\big(k-1|1, -2|n\big) \\ & \quad - 2  \: \big(n-2r-\delta(n-1)\big) \: \Omega_{r+1}\big(k-1|1, -2|n\big),  \\
			\Omega_r\big(0|1, -2|n\big) &= 1 \quad   \text{for all} \:\: r.
		\end{aligned}
	\end{equation}
	Solving \eqref{AU1}, one by one, we get
	\begin{equation}
		\label{AU2} 
		\begin{aligned}
			\Omega_r\big(1|1, -2|n\big) &=  4  \: \big(n-r-1\big) \: \Omega_r\big(0|1, -2|n\big) - 2  \: \big(n-2r-\delta(n-1)\big) \: \Omega_{r+1}\big(0|1, -2|n\big) \\
			&=   4  \: \big(n-r-1\big) \: (1) - 2  \: \big(n-2r-\delta(n-1)\big) \: (1)  \\
			&=  2 \: \big(n+\delta(n-1) - 2\big).  
		\end{aligned}
	\end{equation}
	Continue the process, we get
	\begin{equation}
		\label{AU3} 
		\begin{aligned}
			\Omega_r\big(2|1, -2|n\big) &=  4  \: \big(n-r-2\big) \: \Omega_r\big(1|1, -2|n\big) - 2  \: \big(n-2r-\delta(n-1)\big) \: \Omega_{r+1}\big(1|1, -2|n\big) \\
			&=  2^2 \: \big(n+\delta(n-1) - 2\big)  \: \big(n+\delta(n-1) - 4\big). 
		\end{aligned}
	\end{equation}
	Then again we compute
	\begin{equation}
		\label{AU4} 
		\begin{aligned}
			\Omega_r\big(3|1, -2|n\big) &=  4  \: \big(n-r-3\big) \: \Omega_r\big(2|1, -2|n\big) - 2  \: \big(n-2r-\delta(n-1)\big) \: \Omega_{r+1}\big(2|1, -2|n\big) \\
			&=  2^3 \: \big(n+\delta(n-1) - 2\big)  \: \big(n+\delta(n-1) - 4\big) \: \big(n+\delta(n-1) - 6\big).
		\end{aligned}
	\end{equation}
	
	Finally, we obtain the following result
	\begin{theorem}{}
		\label{AU5}
		The Omega sequence associated with $n$ and the point $(1,-2)$ is given by	
		\begin{equation}
			\label{AU6} 
			\Omega_r\big(k|1, -2|n\big) = \: 2^k \: \prod\limits_{\lambda = 1}^{k}\big(n+\delta(n-1) - 2 \lambda \big).
		\end{equation}
	\end{theorem}
	Put $r=0$, $ k = \left\lfloor \frac{n}{2} \right\rfloor$ in \eqref{AU6}, we immediately get 	
	\begin{theorem}{}
		\label{AU5N}
		\begin{equation}
			\label{AU6N} 
			\Omega_0\big( \left\lfloor \frac{n}{2} \right\rfloor|1, -2|n\big) = \: 2^{ \left\lfloor \frac{n}{2} \right\rfloor} \: \prod\limits_{\lambda = 1}^{ \left\lfloor \frac{n}{2} \right\rfloor}\big(n+\delta(n-1) - 2 \lambda \big).
		\end{equation}
	\end{theorem}
Noting that
\begin{equation}
	\begin{aligned}
		n+\delta(n-1) - 2 \left\lfloor \frac{n}{2} \right\rfloor &= n+\delta(n-1) - 2 \: \: \frac{n-\delta(n)}{2} \\
		&= \delta(n-1) + \delta(n) = 1.
	\end{aligned}
\end{equation}	
Therefore we get the following explicit formula.
		\begin{theorem}{}
		\label{AU5NM}
		\begin{equation}
			\label{AU6NM} 
			\Omega_0\big( \left\lfloor \frac{n}{2} \right\rfloor|1, -2|n\big) = \: 2^{ \left\lfloor \frac{n}{2} \right\rfloor} \: \big(n+\delta(n-1) - 2  \big)\big(n+\delta(n-1) - 4  \big) \dotsc (1).
		\end{equation}
	\end{theorem}
	
			\section{Definition and properties of \texorpdfstring{$\Psi(a,b,n)$}{}}
	In \cite{1}, the author of the current paper first introduced  the following definition
	\begin{definition}
		For any given variables $a,b$, $(a,b) \neq (0,0)$, and for any natural number $n$, we define the sequence $\Psi(a,b,n)= \Psi(n),$ by the following recurrence relation
		\begin{equation}
			\begin{aligned}
				\label{def0}
				\Psi(0)=2, \Psi(1)=1,\Psi(n+1)=(2a-b)^{\delta(n)}\Psi(n) - a \Psi(n-1).
			\end{aligned}
		\end{equation}
	\end{definition}

	\subsection{  Computing  \texorpdfstring{{$\Psi(a,b,n)$}}{} }
	
		The polynomials $\Psi(a,b,n)$ enjoy natural arithmetical and also differential properties and unify many well-known polynomials. One of the methods to compute $\Psi-$sequence, see \cite{1}, for details, is the following explicit formula
	\begin{equation}
		\begin{aligned}
			\label{comp1}	 	
			\Psi(a,b,n) = \frac{(2a-b)^{\lfloor{\frac{n}{2}}\rfloor}} {2^n} \left\{  \left( 1 + \sqrt{ \frac{b+2a}{b-2a}} \right)^n + \left( 1 - \sqrt{ \frac{b+2a}{b-2a}} \right)^n \right\}.
		\end{aligned}	 	
	\end{equation} 	
	
	In \cite{1}, using the formal derivation, we proved the following identity
	\begin{equation}
		\label{00}
		x^{n} + y^{n} =\sum_{i=0}^{\left\lfloor \frac{n}{2} \right\rfloor}(-1)^{i} \frac{n}{n-i}  \binom{n-i}{i} (xy)^i (x+y)^{n-2i}.
	\end{equation}
	From \eqref{comp1}, \eqref{00} we get the proof of the following 
	\begin{theorem}
		\label{comp2}
		For any natural number $n$, the following formula is true
		\begin{equation}
			\label{comp3}
			\Psi(a,b,n) =\sum_{i=0}^{\left\lfloor \frac{n}{2} \right\rfloor}\frac{n}{n-i} \binom{n-i}{i} (-a)^i (2a-b)^{\left\lfloor \frac{n}{2} \right\rfloor - i}.
		\end{equation}
	\end{theorem}
	
	\section{Arithmetic differential properties}
	
	\begin{theorem}
		\label{exp1}
		For any natural number $n$, and any real numbers $a,b, \alpha, \beta$, $ \beta a - \alpha b \neq 0 $, there exist unique polynomials in $a,b, \alpha, \beta$ with integer coefficients, that we call
		$ \Psi\left( \begin{array}{cc|r} a & b & n \\ \alpha & \beta & r \end{array} \right)$, that depend only on $a,b, \alpha, \beta, n,$ and $r$, and satisfy the following polynomial identity
		\begin{equation}
			\label{ex00} 
			\begin{aligned}
				(\beta a - \alpha b)^{\lfloor{\frac{n}{2}}\rfloor} \frac{x^n+y^n}{(x+y)^{\delta(n)}}  = \sum_{r=0}^{\lfloor{\frac{n}{2}}\rfloor}
				\Psi\left(\begin{array}{cc|r} a & b & n \\ \alpha & \beta & r \end{array}\right)
				(\alpha x^2 + \beta xy + \alpha y^2)^{\lfloor{\frac{n}{2}}\rfloor -r} (ax^2+bxy+ay^{2})^{r}. 
			\end{aligned}
		\end{equation}
		Moreover  
		\begin{equation}
			\label{ex000} 
			\begin{aligned}
				\Psi\left(\begin{array}{cc|r} a & b & n \\ \alpha & \beta & 0 \end{array} \right) = \Psi(a,b,n), 
			\end{aligned}
		\end{equation}
	and
		\begin{equation}
			\label{ex111}
			\begin{aligned}			
				\quad \quad \quad\Psi\left(\begin{array}{cc|c} a & b & n \\ \alpha & \beta & \lfloor{\frac{n}{2}}\rfloor \end{array} \right) = (-1)^{\lfloor{\frac{n}{2}}\rfloor} \: \Psi(\alpha,\beta,n). 
			\end{aligned}
		\end{equation}
	\end{theorem}
	\begin{proof}
		We first prove the existence and uniqueness of the integer coefficients of \eqref{ex00}.\\ Multiplying \eqref{00} by $\frac{(\beta a - \alpha b)^{\lfloor{\frac{n}{2}}\rfloor}}{(x+y)^{\delta(n)}}$, we get integer coefficients for expansion \eqref{ex00} after substituting  
		\begin{equation}
			\label{bex1}
			\begin{aligned}
				(\beta a-\alpha b)\: (x+y)^2 = (2a-b)(\alpha x^2 + \beta xy + \alpha y^2) + (\beta- 2\alpha)( a x^2 + b xy + a y^2   ), 
			\end{aligned}
		\end{equation}
		\begin{equation}
			\label{bex2}
			\begin{aligned}
				(\beta a-\alpha b) xy = a (\alpha x^2 + \beta xy + \alpha y^2) + (-\alpha)( a x^2 + b xy + a y^2   ).
			\end{aligned}
		\end{equation}
		The uniqueness of the coefficients come from the fact that $ \alpha x^2 + \beta xy + \alpha y^2 $ and $   a x^2 + b xy + a y^2$ are algebraically independent for $ \beta a - \alpha b \neq 0 $. Put $x=x_0=-b + \sqrt{b^2 - 4a^2}, \:y= y_0=2a$, then $ax_0^2 + bx_0 y_0 + a y_0^2 = 0$. It follows that 
		\begin{equation}
			\begin{aligned}
				\label{comp12}	 	
				\Psi\left(\begin{array}{cc|r} a & b & n \\ \alpha & \beta & 0 \end{array} \right) = \frac{(2a-b)^{\lfloor{\frac{n}{2}}\rfloor}} {2^n} \left\{  \left( 1 + \sqrt{ \frac{b+2a}{b-2a}} \right)^n + \left( 1 - \sqrt{ \frac{b+2a}{b-2a}} \right)^n \right\}.
			\end{aligned}	 	
		\end{equation} 	
		From \eqref{comp12} and \eqref{comp1} we get \eqref{ex000}.
		Now put $x=x_1=-\beta + \sqrt{\beta^2 - 4 \alpha^2}, \:y= y_1=2 \alpha$, then $\alpha x_1^2 + \beta x_1 y_1 + \alpha y_1^2 = 0$. Hence  
		\begin{equation}
			\begin{aligned}
				\label{comp123}	 	
				\Psi\left(\begin{array}{cc|c} a & b & n \\ \alpha & \beta & \lfloor{\frac{n}{2}}\rfloor \end{array} \right) = (-1)^{\lfloor{\frac{n}{2}}\rfloor} \: \frac{(2 \alpha -\beta)^{\lfloor{\frac{n}{2}}\rfloor}} {2^n} \left\{  \left( 1 + \sqrt{ \frac{\beta+2\alpha}{\beta-2\alpha}} \right)^n + \left( 1 - \sqrt{ \frac{\beta+2\alpha}{\beta-2\alpha}} \right)^n \right\}.
			\end{aligned}	 	
		\end{equation} 	
		From \eqref{comp123} and \eqref{comp1} we get \eqref{ex111}. This completes the proof.	\end{proof}	
	
	\begin{theorem}
		\label{Aexp1} For $\beta a - \alpha b \neq 0$,
		the polynomials $\Psi_r(n):=\Psi\left(\begin{array}{cc|c}
			a & b & n \\ \alpha & \beta & r \end{array}\right)$ satisfy
		\begin{align}
			\begin{aligned}
				\label{diff1}
				\big(\alpha \frac{{\partial} }{\partial a} +  \beta \frac{{\partial} }{\partial b}\big)\Psi_r(n) &= - (r+1)\Psi_{r+1}(n), \\
				\big(a \frac{{\partial} }{\partial \alpha} +  b \frac{{\partial} }{\partial \beta} \big) \Psi_r(n) &= - \big(\lfloor{\frac{n}{2}}\rfloor - r + 1 \big)\Psi_{r-1}(n).
			\end{aligned}
		\end{align}
	\end{theorem}
	
	\begin{proof}
		We differentiate \eqref{ex00} with respect to the particular differential operator 
		\[ 		\big(\alpha \frac{{\partial} }{\partial a} +  \beta \frac{{\partial} }{\partial b} \big).    \]
		Noting that
		\begin{equation}
			(\alpha \frac{{\partial} }{\partial a} +  \beta \frac{{\partial} }{\partial b}) (\beta a - \alpha b)^{\lfloor{\frac{n}{2}}\rfloor} \frac{x^n+y^n}{(x+y)^{\delta(n)}} =0,
		\end{equation}
		we get
		\begin{equation}
			\label{p1}
			\begin{aligned}
				0=\Big(\alpha \frac{{\partial} }{\partial a} & +  \beta \frac{{\partial} }{\partial b} \Big) \sum_{r=0}^{\lfloor{\frac{n}{2}}\rfloor}
				\Psi_r(n) (\alpha x^2 + \beta xy + \alpha y^2)^{\lfloor{\frac{n}{2}}\rfloor -r} (ax^2+bxy+ay^{2})^{r}  \\
				&=  \sum_{r=0}^{\lfloor{\frac{n}{2}}\rfloor}
				(\alpha x^2 + \beta xy + \alpha y^2)^{\lfloor{\frac{n}{2}}\rfloor -r} (ax^2+bxy+ay^{2})^{r}\Big(\alpha \frac{{\partial} }{\partial a} +  \beta \frac{{\partial} }{\partial b} \Big)\Psi_r(n)\\
				&+   \sum_{r=0}^{\lfloor{\frac{n}{2}}\rfloor}
				\Psi_r(n)
				(\alpha x^2 + \beta xy + \alpha y^2)^{\lfloor{\frac{n}{2}}\rfloor -r}\Big(\alpha \frac{{\partial} }{\partial a} +  \beta \frac{{\partial} }{\partial b} \Big) (ax^2+bxy+ay^{2})^{r} \\
				&+  \sum_{r=0}^{\lfloor{\frac{n}{2}}\rfloor}
				\Psi_r(n)
				(ax^2+bxy+ay^{2})^{r}\Big(\alpha \frac{{\partial} }{\partial a} +  \beta \frac{{\partial} }{\partial b} \Big) (\alpha x^2 + \beta xy + \alpha y^2)^{\lfloor{\frac{n}{2}}\rfloor -r}.
			\end{aligned}
		\end{equation}
		Consequently, from \eqref{p1}, we obtain the following desirable polynomial expansion
		\begin{equation}
			\label{p2}
			0 = \sum_{r=0}^{\lfloor{\frac{n}{2}}\rfloor} \Big(	(\alpha \frac{{\partial} }{\partial a} +  \beta \frac{{\partial} }{\partial b})\Psi_r(n) + (r+1)\Psi_{r+1}(n) \Big)  (\alpha x^2 + \beta xy + \alpha y^2)^{\lfloor{\frac{n}{2}}\rfloor -r} (ax^2+bxy+ay^{2})^{r}.
		\end{equation}
		As $ \beta a - \alpha b \neq 0 $, the polynomials $(\alpha x^2 + \beta xy + \alpha y^2)$ and $(a x^2 + b xy + a y^2)$ are algebraically independent which means that all of the coefficients of ~\eqref{p2} must vanish. This means that 
		\begin{align}
			\begin{aligned}
				\big(\alpha \frac{{\partial} }{\partial a} +  \beta \frac{{\partial} }{\partial b}\big)\Psi_r(n) + (r+1)\Psi_{r+1}(n)  = 0 \quad   \text{for all} \:\: r.
			\end{aligned}
		\end{align}
		Similarly, we differentiate \eqref{ex00} with respect to the particular differential operator 
		\[ 		\big(a \frac{{\partial} }{\partial \alpha} +  b \frac{{\partial} }{\partial \beta} \big).    \]
		We get
		\begin{align}
			\begin{aligned}
				\big(a \frac{{\partial} }{\partial \alpha} +  b \frac{{\partial} }{\partial \beta} \big) \Psi_r(n) + \big(\lfloor{\frac{n}{2}}\rfloor - r + 1 \big)\Psi_{r-1}(n) = 0  \quad   \text{for all} \:\: r. 
			\end{aligned}
		\end{align}
	This completes the proof.
	\end{proof}
	This result immediately gives the following desirable theorem
	
	\begin{theorem}
		\label{A1} For $\beta a - \alpha b \neq 0$,
		the polynomials $\Psi\left(\begin{array}{cc|c}
			a & b & n \\ \alpha & \beta & r \end{array}\right)$ satisfy
		\begin{align}
			\begin{aligned}
				\label{diff3}
				\Psi\left(\begin{array}{cc|c}
					a & b & n \\ \alpha & \beta & r \end{array} \right) &=  \frac{(-1)^r}{r!}
				\Big(\alpha \frac{{\partial} }{\partial a} + \beta \frac{{\partial}}{\partial b}\Big)^{r} \Psi(a,b,n),  \quad \qquad
			\end{aligned}
		\end{align}
	and
		\begin{align}
			\begin{aligned}
				\label{diff5}
				\Psi\left( \begin{array}{cc|c}
					a & b & n \\ \alpha & \beta & r \end{array} \right) =  \frac{(-1)^r}{(\lfloor{\frac{n}{2}}\rfloor -r)!} \Big(a \frac{{\partial} }{\partial \alpha} + b \frac{{\partial}}{\partial \beta}\Big)^{\lfloor{\frac{n}{2}}\rfloor-r} \Psi(\alpha,\beta,n).  
			\end{aligned}
		\end{align}
	\end{theorem}

\section{ An illustrative example of the calculations of \texorpdfstring{$\Psi$}{}}
	It is desirable to clarify how we apply the methods of Theorem\eqref{A1} to compute the polynomial coefficients $ \Psi\left( \begin{array}{cc|r} a & b & n \\ \alpha & \beta & r \end{array} \right)$. We choose to compute
	\[\Psi\left( \begin{array}{cc|r} a & b & 4 \\ \alpha & \beta & r \end{array} \right) \quad \mbox{for} \quad r=0,1,2=\left\lfloor \frac{4}{2} \right\rfloor. \]
	Then, from Theorem \eqref{exp1}, we get  \[
	\Psi\left( \begin{array}{cc|c}
		a & b & 4 \\ \alpha & \beta & 0 \end{array} \right) =\Psi(a,b,4)= -2a^2+b^2.   \]
	Hence, from Theorem \eqref{A1}, we get 
	\begin{align*}
		\begin{aligned}
			\Psi\left( \begin{array}{cc|c}
				a & b & 4 \\ \alpha & \beta & 1 \end{array} \right)  &=   \frac{-1}{1}\Big(\alpha \frac{{\partial} }{\partial a} +  \beta \frac{{\partial} }{\partial b} \Big)(-2a^2+b^2)= 4a \alpha - 2b \beta , \\
			\Psi\left( \begin{array}{cc|c}
				a & b & 4 \\ \alpha & \beta & 2 \end{array} \right)  &= \frac{-1}{2}  \Big(\alpha \frac{{\partial} }{\partial a} + \beta \frac{{\partial}}{\partial b} \Big) (4a \alpha - 2b \beta) = - 2 \alpha^2 + \beta^2. \\
		\end{aligned}
	\end{align*}
	Then from Theorem \eqref{exp1} we immediately obtain the following polynomial identity
	\begin{align}
		\begin{aligned}
			\label{G1}
			(\beta a - \alpha b)^2 (x^4 + y^4) &= (-2 a^2 + b^2)(\alpha x^2 +\beta xy + \alpha y^2)^2 \\
			&+ (4a \alpha - 2b \beta)(\alpha x^2 +\beta xy + \alpha y^2)(a x^2 +b xy + a y^2) \\
			&+ (-2 \alpha^2 + \beta^2)(a x^2 + b xy + a y^2)^2.
		\end{aligned}
	\end{align}
	Generally, it is desirable to search for values for the parameters $a, b, \alpha, \beta$ that make the middle term, $ \Psi\left( \begin{array}{cc|c} a & b & 4 \\ \alpha & \beta & 1 \end{array} \right)  $ get vanished. Therefore, we put $ \left( \begin{array}{cc} a & b \\ \alpha & \beta \end{array} \right) = \left( \begin{array}{cc} 1 & 1 \\ 1 & 2 \end{array} \right) $. Then the middle coefficient, $4a \alpha - 2b \beta$, of ~\eqref{G1}, is vanished, and we obtain the following special case for a well-known identity in the history of number theory that is used extensively in the  study of equal sums of like powers and in discovering new formulas for Fibonacci numbers
	\begin{align}
		\begin{aligned}
			x^4 + y^4 + (x+y)^4 = 2 (x^2 + xy + y^2)^2.
		\end{aligned}
	\end{align}
	Volume 2, \cite{Dickson}, attributes this special case to C. B. Haldeman (1905), although Proth (1878) used it in passing.
	\section{The fundamental theorem of \texorpdfstring{$\Psi-$}{} sequence}
	Now, put $ r=\lfloor{\frac{n}{2}}\rfloor  $ in equation \eqref{diff3} of Theorem \eqref{A1}, together with Theorem \eqref{exp1}, we get the following immediate consequence
	
	\begin{theorem}{(The fundamental theorem of $\Psi-$sequence)}\\
		\label{Aexp2} For any numbers $a,b, \alpha, \beta  $, $\beta a - \alpha b \neq 0$, and any natural number $n$, we have
		\begin{align}
			\begin{aligned}
				\frac{1}{(\lfloor{\frac{n}{2}}\rfloor)!} \Big(\alpha \frac{{\partial} }{\partial a} + \beta \frac{{\partial}}{\partial b}\Big)^{\lfloor{\frac{n}{2}}\rfloor} \Psi(a,b,n) = \Psi(\alpha,\beta,n).  
			\end{aligned}
		\end{align}
	\end{theorem}		
		Also, it is useful to deduce the following relations. Replace each $\alpha$ and $\beta$ by $\lambda \alpha$ and $\lambda \beta$ respectively in Theorem \eqref{exp1}, we get the following polynomial identity for any $\lambda$
	\begin{equation*}
		\label{}
		\begin{aligned}
			&(\lambda\beta a - \lambda\alpha b)^{\lfloor{\frac{n}{2}}\rfloor} \frac{x^n+y^n}{(x+y)^{\delta(n)}} = \\
			&\sum_{r=0}^{\lfloor{\frac{n}{2}}\rfloor}
			\Psi\left( \begin{array}{cc|r} a & b & n \\ \lambda\alpha & \lambda\beta & r \end{array} \right)
			(\lambda\alpha x^2 + \lambda\beta xy + \lambda \alpha y^2)^{\lfloor{\frac{n}{2}}\rfloor -r} (ax^2+bxy+ay^{2})^{r}.
		\end{aligned}
	\end{equation*}
	Then
	\begin{equation}
		\label{V1}
		\begin{aligned}
			&(\beta a - \alpha b)^{\lfloor{\frac{n}{2}}\rfloor} \frac{x^n+y^n}{(x+y)^{\delta(n)}} = \\ &\sum_{r=0}^{\lfloor{\frac{n}{2}}\rfloor} \lambda^{ -r}
			\Psi\left( \begin{array}{cc|r} a & b & n \\ \lambda\alpha & \lambda\beta & r \end{array} \right)
			(\alpha x^2 + \beta xy +  \alpha y^2)^{\lfloor{\frac{n}{2}}\rfloor -r} (ax^2+bxy+ay^{2})^{r}.
		\end{aligned}
	\end{equation}
	Comparing \eqref{V1} with \eqref{exp1}, we obtain
	\[     \lambda^{ -r}
	\Psi\left( \begin{array}{cc|r} a & b & n \\ \lambda\alpha & \lambda\beta & r \end{array} \right) =     \Psi\left( \begin{array}{cc|r} a & b & n \\ \alpha & \beta & r \end{array} \right).       \]
	Similarly, we can prove the following useful relations.
	
	\begin{theorem}
		\label{W11}
		For any numbers $a, b, \alpha, \beta, \beta a - \alpha b \neq 0, \lambda, r,  n $,  we get
		\begin{equation}
			\label{W22}
			\begin{aligned}
				\Psi\left( \begin{array}{cc|c}
					a & b & n \\ \lambda \alpha & \lambda \beta & r \end{array} \right)
				&=\lambda^{r} \Psi\left( \begin{array}{cc|c}
					a & b & n \\ \alpha & \beta & r \end{array} \right),  \\  
				\Psi\left( \begin{array}{cc|c}
					\lambda a &\lambda b & n \\ \alpha & \beta & r \end{array} \right)
				&=\lambda^{\lfloor{\frac{n}{2}}\rfloor - r}
				\Psi\left(\begin{array}{cc|c}
					a & b & n \\ \alpha & \beta & r \end{array} \right), \\
				\Psi\left( \begin{array}{cc|c}
					a & b & n \\ \alpha & \beta & r \end{array} \right) &=
				(-1)^{\lfloor{\frac{n}{2}} \rfloor}
				\Psi\left( \begin{array}{cc|c}
					\alpha & \beta  & n \\  a & b &\lfloor{\frac{n}{2}} \rfloor - r \end{array} \right),
				\\
				\text{and}\quad \quad \quad \quad  \quad \quad \quad \quad \quad \quad \quad \quad \\
				\lambda^{\lfloor{\frac{n}{2}}\rfloor}\Psi(a,b,n) &= \Psi(\lambda a,\lambda  b,n).  		
			\end{aligned}
		\end{equation}
	\end{theorem}
	
	\section{The \texorpdfstring{$\Psi-$}{}representation for the  \texorpdfstring{$\Psi-$}{}sequence}
	We now ready to prove the following theorem.
	\begin{theorem}
		\label{Ready00}
		For any $a,b,\alpha,\beta, \theta, n$, $\beta a - \alpha b \neq 0,$ the following identities are true
		\begin{equation}
			\label{Ready1}
			\sum_{r=0}^{\lfloor{\frac{n}{2}}\rfloor}  \Psi\left( \begin{array}{cc|r} a & b & n \\ \alpha & \beta & r \end{array} \right) \theta^r = \Psi(a-\alpha \theta , b - \beta \theta, n).
		\end{equation}
		
	\end{theorem}
	\begin{proof}
		Define $q_1:=\alpha x^2+\beta xy +\alpha y^2$ and $q_2:=a x^2+b xy +a y^2$ and
		\[ \Lambda_\theta :=\theta q_1 - q_2 = (\alpha \theta - a)x^2 + (\beta \theta - b)xy + (\alpha \theta - a)y^2 .\] 	
		From Theorem \eqref{exp1}, we know that
		\[
		(\beta a - \alpha b)^{\lfloor{\frac{n}{2}}\rfloor} \frac{x^n+y^n}{(x+y)^{\delta(n)}} = \sum_{r=0}^{\lfloor{\frac{n}{2}}\rfloor}  \Psi\left( \begin{array}{cc|r} a & b & n \\ \alpha & \beta & r \end{array} \right) (q_1)^{\lfloor{\frac{n}{2}}\rfloor -r} (q_2)^{r}. \]
		As $q_2 \equiv \theta q_1  \pmod{\Lambda_\theta}$, we get
		\begin{equation}
			\label{Ready2}
			(\beta a - \alpha b)^{\lfloor{\frac{n}{2}}\rfloor} \frac{x^n+y^n}{(x+y)^{\delta(n)}} \equiv q_1^{\lfloor{\frac{n}{2}}\rfloor} \sum_{r=0}^{\lfloor{\frac{n}{2}}\rfloor}  \Psi\left( \begin{array}{cc|r} a & b & n \\ \alpha & \beta & r \end{array} \right) \theta^r   \pmod{\Lambda_\theta}.
		\end{equation}
		Replace each of $a,b$ by $\alpha \theta - a, \beta \theta - b$ respectively,  in Theorem \eqref{exp1}, we obtain
		\begin{equation}
			\label{}
			(\beta [\alpha \theta - a] -
			\alpha[\beta \theta - b])^{\lfloor{\frac{n}{2}}\rfloor} \frac{x^n+y^n}{(x+y)^{\delta(n)}} \equiv
			\Psi(\alpha \theta - a, \beta \theta - b, n)  q_1^{\lfloor{\frac{n}{2}}\rfloor} \pmod{\Lambda_\theta}.
		\end{equation}	
		As $\beta [\alpha \theta - a] -
		\alpha[\beta \theta - b] = - (\beta a - \alpha b)$, and noting from Theorem \eqref{W11} that
		\[ (-1)^{\lfloor{\frac{n}{2}}\rfloor} \Psi(\alpha \theta - a, \beta \theta - b, n) =   \Psi(a-\alpha \theta , b - \beta \theta, n), \]
		we immediately get the following congruence
		\begin{equation}
			\label{Ready3}
			(  \beta a - \alpha b  )^{\lfloor{\frac{n}{2}}\rfloor} \frac{x^n+y^n}{(x+y)^{\delta(n)}} \equiv
			\Psi(a-\alpha \theta , b - \beta \theta, n) q_1^{\lfloor{\frac{n}{2}}\rfloor} \pmod{\Lambda_\theta}.
		\end{equation}
		Now, subtracting \eqref{Ready2} and\eqref{Ready3}, we obtain
		\begin{equation}
			\label{Ready4}
			0 \equiv \Big(
			\sum_{r=0}^{\lfloor{\frac{n}{2}}\rfloor}  \Psi\left( \begin{array}{cc|r} a & b & n \\ \alpha & \beta & r \end{array} \right) \theta^r -
			\Psi(a-\alpha \theta , b - \beta \theta, n) \Big) q_1^{\lfloor{\frac{n}{2}}\rfloor} \pmod{\Lambda_\theta}.
		\end{equation}
		As the congruence \eqref{Ready4} is true for any $x,y$, and as $(\beta \theta - b) \alpha -  (\alpha \theta - a) \beta = \beta a - \alpha b \neq 0$, then the binary quadratic forms $\Lambda_\theta$ and $ q_1 $ are algebraically independent. This immediately leads to
		\begin{equation}
			\label{Ready5}
			0= \sum_{r=0}^{\lfloor{\frac{n}{2}}\rfloor}  \Psi\left( \begin{array}{cc|r} a & b & n \\ \alpha & \beta & r \end{array} \right) \theta^r -
			\Psi(a-\alpha \theta , b - \beta \theta, n).
		\end{equation}
		Hence we obtained the proof of \eqref{Ready1}. This completes the proof of Theorem \eqref{Ready00}.  \end{proof}
	\section{Specialization and lifting}
	The following desirable generalization is important
	
	\begin{theorem}
		\label{Ready6}
		For any $a,b,\alpha,\beta, \eta, \xi, n$, $\beta a - \alpha b \neq 0,$ the following identities are true
		\begin{equation}
			\label{Ready66}
			\begin{aligned}
				\sum_{r=0}^{\lfloor{\frac{n}{2}}\rfloor}  \Psi\left( \begin{array}{cc|r} a & b & n \\ \alpha & \beta & r \end{array} \right)  \xi^{\lfloor{\frac{n}{2}}\rfloor  - r}  \eta^r =\Psi(a \xi-\alpha \eta, b \xi - \beta \eta, n).
			\end{aligned}
		\end{equation}
	\end{theorem}
	\begin{proof}
		Without loss of generality, let $\xi \neq 0$. We obtain the proof by replacing each $\theta$ in equation \eqref{Ready1} of Theorem \eqref{Ready00} by $\frac{\eta}{\xi},$ and multiplying each side by $\xi^{\lfloor{\frac{n}{2}}\rfloor},$ and  noting from Theorem \eqref{W11} that 
		\[\xi^{\lfloor{\frac{n}{2}}\rfloor}\Psi(a-\alpha \frac{\eta}{\xi} , b - \beta \frac{\eta}{\xi}, n) = \Psi(a \xi-\alpha \eta, b \xi - \beta \eta, n). \] 
	\end{proof}
	Replacing $\theta$ by $\pm 1$ in \eqref{Ready1}, we obtain the following desirable special cases
	\begin{theorem}
		\label{sum}
		For any $a,b,\alpha,\beta, n$, $\beta a - \alpha b \neq 0$ the following identities are true
		\begin{align}
			\begin{aligned}
				\label{theta equal 1}
				\centering	 	
				\sum_{r=0}^{\lfloor{\frac{n}{2}}\rfloor}  \Psi\left( \begin{array}{cc|r} a & b & n \\ \alpha & \beta & r \end{array} \right) &= \Psi(a-\alpha , b - \beta , n), \\ 	 	
				\sum_{r=0}^{\lfloor{\frac{n}{2}}\rfloor}  \Psi\left( \begin{array}{cc|r} a & b & n \\ \alpha & \beta & r \end{array} \right) (-1)^{r} &= \Psi(a + \alpha , b + \beta , n). 
			\end{aligned}
		\end{align}
	\end{theorem}

	\subsection{More generalizations}
	Now we can generalize Theorem \eqref{Ready00} by applying the following specific differential map
	\[ \Big( - \frac{{\partial}}{\partial \theta} \Big)^{k}   \]
	on equation \eqref{Ready1}, and  noting that 
	
	\begin{equation*}
		\begin{aligned}
			\Big( - \frac{{\partial}}{\partial \theta} \Big)^{k} &
			\Psi(a-\alpha \theta , b - \beta \theta, n) 
			= (-1)^{k} (k!) \Psi\left( \begin{array}{cc|r} a-\alpha \theta & b - \beta \theta & n \\ \alpha & \beta & k\end{array} \right).	
		\end{aligned}
	\end{equation*}
	Hence we immediately obtain the following desirable generalization
	\begin{theorem}
		\label{Ready7}
		For any $n,k, a,b,\alpha,\beta, \theta, n$, $\beta a - \alpha b \neq 0,$ the following identity is true
		\begin{equation}
			\label{Ready8}
			\sum_{r=k}^{\lfloor{\frac{n}{2}}\rfloor} \binom{r}{k}  \Psi\left( \begin{array}{cc|r} a & b & n \\ \alpha & \beta & r \end{array} \right) \theta^{r-k} = \Psi\left( \begin{array}{cc|r} a-\alpha \theta  & b-\beta \theta  & n \\ \alpha & \beta & k \end{array} \right).
		\end{equation}
		
	\end{theorem}
	Again, without loss of generality, let $\xi \neq 0$. By replacing each $\theta$ in \eqref{Ready8} by $\frac{\eta}{\xi}$, and multiplying each side by $\xi^{\lfloor{\frac{n}{2}}\rfloor -k}$, and noting the properties of $\Psi$ of Theorem \eqref{W11} that
		
		\begin{equation*} \xi^{\lfloor{\frac{n}{2}}\rfloor -k} \Psi\left( \begin{array}{cc|r} a-\alpha \frac{\eta}{\xi}  & b-\beta \frac{\eta}{\xi}  & n \\ \alpha & \beta & k \end{array}     \right)  =  \Psi\left( \begin{array}{cc|r} a\xi-\alpha \eta  & b\xi-\beta \eta  & n \\ \alpha & \beta & k \end{array} \right),        		\end{equation*}
		
	we obtain the following generalization for Theorem \eqref{Ready7}
	\begin{theorem}
		\label{Ready9}
		For any $n,k,a,b,\alpha,\beta, \theta, n$, $\beta a - \alpha b \neq 0,$ the following identity is true
		\begin{align}
			\begin{aligned}
				\label{Ready10}
				\sum_{r=k}^{\lfloor{\frac{n}{2}}\rfloor} \binom{r}{k}  \Psi\left( \begin{array}{cc|r} a & b & n \\ \alpha & \beta & r \end{array} \right) \xi^{\lfloor{\frac{n}{2}}\rfloor  - r}  \eta^{r-k} = \Psi\left( \begin{array}{cc|r} a\xi-\alpha \eta  & b\xi-\beta \eta  & n \\ \alpha & \beta & k \end{array} \right).
			\end{aligned}
		\end{align}
	\end{theorem}
	\section{The \texorpdfstring{$\Psi-$}{}representation for sums of powers}
	Now, put $\xi = \alpha x^2 +\beta xy + \alpha y^2$ and $ \eta = a x^2 +bxy +a y^2$ in equation \eqref{Ready66} of Theorem \eqref{Ready6}, we obtain
	\begin{equation}
		\begin{aligned}
			\label{WW1}
			\sum_{r=0}^{\lfloor{\frac{n}{2}}\rfloor}  \Psi\left( \begin{array}{cc|r} a & b & n \\ \alpha & \beta & r \end{array} \right) & (\alpha x^2 +\beta xy + \alpha y^2)^{\lfloor{\frac{n}{2}}\rfloor  - r}  (  a x^2 +bxy +a y^2 )^r \\
			& = (  \beta a - \alpha b  )^{\lfloor{\frac{n}{2}}\rfloor} \Psi(xy,-x^2-y^2, n).
		\end{aligned}
	\end{equation}
	Now, from \eqref{exp1}, we get
	\begin{equation}
		\begin{aligned}
			\label{WW2}
			\sum_{r=0}^{\lfloor{\frac{n}{2}}\rfloor}  \Psi\left( \begin{array}{cc|r} a & b & n \\ \alpha & \beta & r \end{array} \right) & (\alpha x^2 +\beta xy + \alpha y^2)^{\lfloor{\frac{n}{2}}\rfloor  - r}  (  a x^2 +bxy +a y^2 )^r \\
			&= (  \beta a - \alpha b  )^{\lfloor{\frac{n}{2}}\rfloor}\frac{x^n+y^n}{(x+y)^{\delta(n)}}.
		\end{aligned}
	\end{equation}
	From \eqref{WW1}, \eqref{WW2}, we get the following desirable $\Psi-$representation for the sums of powers 	
	\begin{theorem}{(The $\Psi-$representation for sums of powers)}
		\label{WW3}
		For any natural number $n$, the $\Psi-$polynomial satisfy the following identity
		\begin{equation}
			\label{WW4}
			\begin{aligned}
				\Psi(xy,-x^2-y^2,n) &= \frac{x^n+y^n}{(x+y)^{\delta(n)}}.   \\
			\end{aligned}
		\end{equation}
	\end{theorem}
	
	\section{The representation of \texorpdfstring{$\Psi$}{} in terms of \texorpdfstring{$\Omega$}{} sequence }
	
	\begin{theorem}{}
		\label{FD1} 
		For any $n,k,a,b,\alpha,\beta, \theta, n$, $\beta a - \alpha b \neq 0,$ the following expansion is true
		\begin{equation}
			\label{FD2}
			\begin{aligned}
				\Psi\left(\begin{array}{cc|c} a & b & n \\ \alpha & \beta & k \end{array} \right) = \sum_{r=0}^{\left\lfloor \frac{n}{2} \right\rfloor -k} \frac{(-1)^k}{k!}  \: \lambda_r(k|\alpha, \beta|n) a^r (2a-b)^{\left\lfloor \frac{n}{2} \right\rfloor -k- r}
			\end{aligned}
		\end{equation}
		where the numbers $\lambda_r(k|\alpha, \beta|n)$ are integers and divisible by $k!$ and satisfy the double-indexed recurrence relation 		
		
		\begin{equation}
			\label{FD3}
			\begin{aligned}
				\begin{cases}
					&        \lambda_r(k|\alpha, \beta|n)= 
					\:	(2\alpha - \beta) \:
					\Big(\left\lfloor \frac{n}{2} \right\rfloor -k- r +1 \Big)
					\: \lambda_r(k-1|\alpha, \beta|n)      \\	  
					&\qquad \qquad \qquad  \qquad \qquad \qquad  \qquad \qquad \:\: \: + \quad  \alpha \: (r+1) \: \lambda_{r+1}(k-1|\alpha, \beta|n)         	    \:     \\ 
					&\: \lambda_r(\:0\:|\:\alpha, \beta \:|\:n\:) = (-1)^r  \frac{n}{n-r} \binom{n-r}{r}   \\
				\end{cases} 
			\end{aligned}
		\end{equation}
	\end{theorem}
		\begin{proof}  
		From Theorem\eqref{A1}, we know that 
		\[ \Psi\left(\begin{array}{cc|c} a & b & n \\ \alpha & \beta & k \end{array} \right) =  \frac{(-1)^k}{k!} \Big(\alpha \frac{{\partial} }{\partial a} + \beta \frac{{\partial}}{\partial b}\Big)^{k} \Psi(a,b,n). \]			
		Then from equation \eqref{comp3}, it follows  	
		\begin{equation}
			\label{Q1}
			\begin{aligned}
				\Psi\left(\begin{array}{cc|c} a & b & n \\ \alpha & \beta & k \end{array} \right)  &= \frac{(-1)^k}{k!} \Big(\alpha \frac{{\partial} }{\partial a} + \beta \frac{{\partial}}{\partial b}\Big)^{k}  \sum_{r=0}^{\left\lfloor \frac{n}{2} \right\rfloor}\frac{n}{n-r} \binom{n-r}{r} (-a)^r (2a-b)^{\left\lfloor \frac{n}{2} \right\rfloor - r} \\
				&= \frac{(-1)^k}{k!} \sum_{r=0}^{\left\lfloor \frac{n}{2} \right\rfloor} (-1)^r  \frac{n}{n-r} \binom{n-r}{r}  \Big(\alpha \frac{{\partial} }{\partial a} + \beta \frac{{\partial}}{\partial b}\Big)^{k}   a^r (2a-b)^{\left\lfloor \frac{n}{2} \right\rfloor - r}. 
			\end{aligned}
		\end{equation}
		From \eqref{Q1}, and for $0 \leq k+r \leq \left\lfloor \frac{n}{2} \right\rfloor$, there exist integers $\lambda_r(k|\alpha, \beta|n)$, which  are divisible by $k!$, and depend only on the numbers $r,k,\alpha, \beta,n$ (and independent on $a,b$), such that   
		\begin{equation}
			\label{Q2}
			\begin{aligned}
				\Psi\left(\begin{array}{cc|c} a & b & n \\ \alpha & \beta & k \end{array} \right) = \frac{(-1)^k}{k!} \sum_{r=0}^{\left\lfloor \frac{n}{2} \right\rfloor -k}  \lambda_r(k|\alpha, \beta|n) a^r (2a-b)^{\left\lfloor \frac{n}{2} \right\rfloor -k- r}. \\
			\end{aligned}
		\end{equation}
		This proves \eqref{FD2}. Now, to study the coefficients $\lambda_r(k|\alpha, \beta|n)$, we need to compute the recurrence relation that arise up easily once we notice, from Theorem \eqref{Aexp1}, that the following differential property of $\Psi$ for any non-negative integer $k$ 
		
		\begin{align}
			\begin{aligned}
				\label{Q4}
				\big(\alpha \frac{{\partial} }{\partial a} +  \beta \frac{{\partial} }{\partial b}\big)	\Psi\left(\begin{array}{cc|c} a & b & n \\ \alpha & \beta & k \end{array} \right) = - (k+1)	\Psi\left(\begin{array}{cc|c} a & b & n \\ \alpha & \beta & k+1 \end{array} \right).   
			\end{aligned}
		\end{align}	
		Then from \eqref{Q2}, \eqref{Q4}, we get 	
		\begin{equation}
			\label{Q5}
			\begin{aligned}
				\big(\alpha \frac{{\partial} }{\partial a} +  \beta \frac{{\partial} }{\partial b}\big)		\frac{(-1)^k}{k!} &\sum_{r=0}^{\left\lfloor \frac{n}{2} \right\rfloor -k}  \lambda_r(k|\alpha, \beta|n) a^r (2a-b)^{\left\lfloor \frac{n}{2} \right\rfloor -k- r} \\ &= - (k+1)  \frac{(-1)^{(k+1)}}{(k+1)!} \sum_{r=0}^{\left\lfloor \frac{n}{2} \right\rfloor -k -1}  \lambda_r(k+1|\alpha, \beta|n) a^r (2a-b)^{\left\lfloor \frac{n}{2} \right\rfloor -k-1- r}.
			\end{aligned}
		\end{equation}	
		Simplifying, and noting that the coefficients $\lambda_r(k|\alpha, \beta|n)$ are independent on $a,b$, we get
		\[ \big(\alpha \frac{{\partial} }{\partial a} +  \beta \frac{{\partial} }{\partial b}\big)  \lambda_r(k|\alpha, \beta|n) =0.                  \] 
		Therefore
		\begin{equation}
			\label{Q6}
			\begin{aligned}
				\sum_{r=0}^{\left\lfloor \frac{n}{2} \right\rfloor -k}  \lambda_r(k|\alpha, \beta|n)        &  	\big(\alpha \frac{{\partial} }{\partial a} +  \beta \frac{{\partial} }{\partial b}\big)	            a^r (2a-b)^{\left\lfloor \frac{n}{2} \right\rfloor -k- r} \\ &=  \sum_{r=0}^{\left\lfloor \frac{n}{2} \right\rfloor -k -1}  \lambda_r(k+1|\alpha, \beta|n) a^r (2a-b)^{\left\lfloor \frac{n}{2} \right\rfloor -k- r -1}.
			\end{aligned}
		\end{equation}		
		Hence
		\begin{equation}
			\label{Q7}
			\begin{aligned}
				\sum_{r=0}^{\left\lfloor  \frac{n}{2} \right\rfloor -k -1}  \alpha &\: (r+1) \: \lambda_{r+1}(k|\alpha, \beta|n)         	    \:     a^r \:  (2a-b)^{\left\lfloor \frac{n}{2} \right\rfloor -k- r-1} \\ 
				+	\sum_{r=0}^{\left\lfloor \frac{n}{2} \right\rfloor -k -1} & (2\alpha - \beta)
				\Big(\left\lfloor \frac{n}{2} \right\rfloor -k- r \Big)
				\: \lambda_r(k|\alpha, \beta|n)      	    \:     a^r \: (2a-b)^{\left\lfloor \frac{n}{2} \right\rfloor -k- r-1} \\
				=  \sum_{r=0}^{\left\lfloor \frac{n}{2} \right\rfloor -k -1} & \lambda_r(k+1|\alpha, \beta|n) \: a^r \: (2a-b)^{\left\lfloor \frac{n}{2} \right\rfloor -k- r-1}.
			\end{aligned}
		\end{equation}	
		From \eqref{Q7}, comparing the coefficients, and noting that $a, 2a-b$ are algebraically independent, we immediately get 
	
		\begin{equation}
			\label{Q9}
			\begin{aligned}
				\alpha \: (r+1)  \: \lambda_{r+1}(k|\alpha, \beta|n)  
				+ \:	(2\alpha - \beta) \:
				\Big(\left\lfloor \frac{n}{2} \right\rfloor -k- r \Big)&
				\:  \lambda_r(k|\alpha, \beta|n)       	  \\
				&\qquad  =   \lambda_r(k+1|\alpha, \beta|n).
			\end{aligned}
		\end{equation}
		
		Now the initial value for $\lambda_r(\:k\:|\alpha, \beta \:|\:n\:)$, that corresponds to $k=0$, is given by
		
		\begin{equation}
			\label{FD2Q}
			\begin{aligned}
				\Psi\left(\begin{array}{cc|c} a & b & n \\ \alpha & \beta & 0 \end{array} \right) = \sum_{r=0}^{\left\lfloor \frac{n}{2} \right\rfloor -k} \: \lambda_r(0|\alpha, \beta|n) a^r (2a-b)^{\left\lfloor \frac{n}{2} \right\rfloor - r}.
			\end{aligned}
		\end{equation}
	Then, from 	\eqref{ex000}, we get 
		\begin{equation}
		\label{FD2QQ}
		\begin{aligned}
			\Psi(a,b,n)= \sum_{r=0}^{\left\lfloor \frac{n}{2} \right\rfloor -k}  \: \lambda_r(0|\alpha, \beta|n) a^r (2a-b)^{\left\lfloor \frac{n}{2} \right\rfloor - r}.
		\end{aligned}		
\end{equation}		
Consequently, from \eqref{comp3}, we immediately get 		
   \begin{equation}
			\label{Q3}
			\begin{aligned}
				\lambda_r(\:0\:|\alpha, \beta \:|\:n\:) = (-1)^r  \frac{n}{n-r} \binom{n-r}{r}, 
			\end{aligned}
		\end{equation}
which completes the proof.	
			\end{proof}
	
	\begin{theorem}
		\label{H1} 
		For any $n,r, k,\alpha,\beta, n$, the following relation is true

		\begin{equation}
			\label{H2} 
			\begin{aligned}
				\lambda_r(k|\alpha, \beta|n)= (-1)^r \: \frac{\: n \: \: (n-r-k-1)! \: (\left\lfloor \frac{n}{2} \right\rfloor - r )! \: }{\: (n-2r)! \:\: r! \:\:  (\left\lfloor \frac{n}{2} \right\rfloor - r -k)! \: }	\:  \Omega_r\big(k|\:\alpha, \beta \: | n \big).         		
			\end{aligned}
		\end{equation}
	\end{theorem}
	
	\begin{proof}
		To prove \eqref{H2}, define $\bar{\Omega}_r\big(k|\:\alpha, \beta \: | n \big)$ as following
		
		\begin{equation}
			\label{H3} 
			\begin{aligned}
				\lambda_r(k|\alpha, \beta|n)= (-1)^r \: \frac{\: n \: \: (n-r-k-1)! \: (\left\lfloor \frac{n}{2} \right\rfloor - r )! \: }{\: (n-2r)! \:\: r! \:\:  (\left\lfloor \frac{n}{2} \right\rfloor - r -k)! \: }	\:  \bar{\Omega}_r\big(k|\:\alpha, \beta \: | n \big).         		
			\end{aligned}
		\end{equation}
		We need to prove that 
		\begin{equation}
			\label{H4} 
			\begin{aligned}
				\bar{\Omega}_r\big(k|\:\alpha, \beta \: | n \big) \: = \:	\Omega_r\big(k|\:\alpha, \beta \: | n \big),         		
			\end{aligned}
		\end{equation}
		as following. From \eqref{H3}, 	\eqref{FD3}, we get
		
		\begin{equation}
			\label{H5} 
			\begin{aligned}
				& (-1)^r \: \frac{\: n \: \: (n-r-k-1)! \: (\left\lfloor \frac{n}{2} \right\rfloor - r )! \: }{\: (n-2r)! \:\: r! \:\:  (\left\lfloor \frac{n}{2} \right\rfloor - r -k)! \: }	\:  \bar{\Omega}_r\big(k|\:\alpha, \beta \: | n \big) \\   &= \: 
				(2\alpha - \beta) \:
				\Big(\left\lfloor \frac{n}{2} \right\rfloor -k- r +1 \Big)  (-1)^r \: \frac{\: n \: \: (n-r-k)! \: (\left\lfloor \frac{n}{2} \right\rfloor - r )! \: }{\: (n-2r)! \:\: r! \:\:  (\left\lfloor \frac{n}{2} \right\rfloor - r -k +1)! \: }	\:  \bar{\Omega}_r\big(k-1|\:\alpha, \beta \: | n \big) \\
				& + \quad  \alpha \: (r+1) \: 
				(-1)^{r+1} \: \frac{\: n \: \: (n-r-k-1)! \: (\left\lfloor \frac{n}{2} \right\rfloor - r -1 )! \: }{\: (n-2r -2)! \:\: (r+1)! \:\:  (\left\lfloor \frac{n}{2} \right\rfloor - r -k )! \: }	\:  \bar{\Omega}_{r+1}\big(k-1|\:\alpha, \beta \: | n \big).
			\end{aligned}
		\end{equation}	
		Simplifying again, we get 
		\begin{equation}
			\label{H6} 
			\begin{aligned}
				& \:  \bar{\Omega}_r\big(k|\:\alpha, \beta \: | n \big) \\   &= \: 
				(2\alpha - \beta) \: (n-r-k)	\:  \bar{\Omega}_r\big(k-1|\:\alpha, \beta \: | n \big) \\
				& - \quad  \alpha \:
				\: \frac{\: \: (n-2r) \: (n-2r-1) \:}{\: (\left\lfloor \frac{n}{2} \right\rfloor - r )! \: }	\:  \bar{\Omega}_{r+1}\big(k-1|\:\alpha, \beta \: | n \big).
			\end{aligned}
		\end{equation}		
		If $\delta(n)=0$ then $ \delta(n-1)=1$, and if  	$\delta(n)=1$ then $\delta(n-1)=0$. Therefore, for either case, we get the following 
		\begin{equation}
			\label{H7} 
			\begin{aligned}
				\: \frac{\: \: (n-2r) \: (n-2r-1) \:}{\: (\left\lfloor \frac{n}{2} \right\rfloor - r )! \: }	\: &=  \frac{(n-2r-\delta(n)) \:(n-2r-\delta(n-1)) \: }{ \frac{n-\delta(n)}{2}\: - \:r} \\
				&=  2 \:\frac{(n-2r-\delta(n) \:(n-2r-\delta(n-1)) \: }{n-2r-\delta(n) } \\	
				&=   2 \: 	(n-2r-\delta(n-1)).
			\end{aligned}
		\end{equation}
		Hence, from \eqref{H6}, \eqref{H7}, we get
		
		\begin{equation}
			\label{H8} 
			\begin{aligned}
				& \:  \bar{\Omega}_r\big(k|\:\alpha, \beta \: | n \big) \\   &= \: 
				(2\alpha - \beta) \: (n-r-k)	\:  \bar{\Omega}_r\big(k-1|\:\alpha, \beta \: | n \big) \\
				& - \: 2 \:  \alpha \:
				\: 	(n-2r-\delta(n-1))	\:  \bar{\Omega}_{r+1}\big(k-1|\:\alpha, \beta \: | n \big).
			\end{aligned}
		\end{equation}	
		Now, it remains to compute the initial value 
		\[ \bar{\Omega}_r\big(0|\:\alpha, \beta \: | n \big),   \]
		as following.	Put $k=0$ in \eqref{H3}, and noting \eqref{FD3},   we get 
		
		\begin{equation}
			\label{H9} 
			\begin{aligned}
				(-1)^r  \frac{n}{n-r} \binom{n-r}{r}    = (-1)^r \: \frac{\: n \: \: (n-r-1)! \: (\left\lfloor \frac{n}{2} \right\rfloor - r )! \: }{\: (n-2r)! \:\: r! \:\:  (\left\lfloor \frac{n}{2} \right\rfloor - r )! \: }	\:  \bar{\Omega}_r\big( 0 \:|\:\alpha, \beta \: | n \big).         		
			\end{aligned}
		\end{equation}
		Consequently, for any $r$, we get 	
		\begin{equation}
			\label{H10} 
			\begin{aligned}
				1   = \:  \bar{\Omega}_r\big(0\:|\:\alpha, \beta\: | n \big).         		
			\end{aligned}
		\end{equation}
		From \eqref{H10}, \eqref{H8}, and from definition \eqref{omegaDef} of Omega sequence, we immediately conclude
		\begin{equation}
			\label{H11} 
			\begin{aligned}
				\bar{\Omega}_r\big(k|\:\alpha, \beta \: | n \big) \: = \:	\Omega_r\big(k|\:\alpha, \beta\: | n \big).         		
			\end{aligned}
		\end{equation}
		This completes the proof of Theorem \eqref{H1}.	
		\end{proof}
 Therefore, from Theorem \eqref{FD1} and Theorem \eqref{H1}, we get the following representation for $\Psi\left( \begin{array}{cc|r} a & b & n \\ \alpha & \beta & k \end{array} \right)    $ in terms of $\Omega-$sequence, which is quite desirable.
	\begin{theorem}{(The representation of $\Psi-$sequence in terms of $\Omega-$sequence)}\\
		\label{F11} 
For any numbers $a,b,\alpha,\beta, n$, $\beta a - \alpha b \neq 0$, we get the following expansion  		
		
		\begin{equation}
			\label{F1100} 
			\begin{aligned}
				&\quad \quad \Psi\left( \begin{array}{cc|r} a & b & n \\ \alpha & \beta & k \end{array} \right)   \\
				&=  \sum_{r=0}^{\lfloor{\frac{n}{2}}\rfloor - k} (-1)^{r+k} \: \frac{\: (n-r-k-1)! \: \:n \: \:}{(n-2r)! \: r!}
				\left(\begin{array}{c} \lfloor{\frac{n}{2}}\rfloor - r \\ k \end{array}\right)   \:	\Omega_r\big(k|\:\alpha, \beta \: | n \big) \:\:
				a^{r} \: (2a-b)^{\lfloor{\frac{n}{2}}\rfloor -k -r},
			\end{aligned}
		\end{equation}
		where the coefficients 
		\begin{equation}
			\label{F22} 
			\begin{aligned}
				(-1)^{r+k} \: \frac{\: (n-r-k-1)! \: \:n \: \:}{(n-2r)! \: r!}
				\left(\begin{array}{c} \lfloor{\frac{n}{2}}\rfloor - r \\ k \end{array}\right)   \:	\Omega_r\big(k|\:\alpha, \beta \: | n \big) 
			\end{aligned}
		\end{equation}
		are integers. 	
	\end{theorem}

	Now, we get the following desirable theorem.
	\begin{theorem}{ (The first fundamental theorem of $\Omega-$sequence) }\\
		\label{k0} 
For any numbers $\alpha,\beta, n$, $(\alpha, \beta)\neq (0,0)$, the ratio 

\begin{equation}
	\label{k00N} 
	\begin{aligned}
	 \: \frac{\:	\Omega_0\big(\left\lfloor \frac{n}{2} \right\rfloor|\:\alpha, \beta\: | n \big)  \:}{\:(n-1)(n-2) \cdots  (n-\left\lfloor \frac{n}{2} \right\rfloor) \: }   		
	\end{aligned}
\end{equation}
is integer. Moreover this ratio gives $\Psi(\alpha, \beta,n)$. Namely
		\begin{equation}
			\label{k00} 
			\begin{aligned}
				\Psi(\alpha, \beta, n) \:= \: \frac{\:	\Omega_0\big(\left\lfloor \frac{n}{2} \right\rfloor|\:\alpha, \beta\: | n \big)  \:}{\:(n-1)(n-2) \cdots  (n-\left\lfloor \frac{n}{2} \right\rfloor) \: }.    		
			\end{aligned}
		\end{equation}
	\end{theorem}
	
	\begin{proof}
		To deduce the formula \eqref{k00} for  \[	\Omega_0\big(\left\lfloor \frac{n}{2} \right\rfloor\:|\:\alpha, \beta \: | n \big),       \]	
		we put $k=\left\lfloor \frac{n}{2} \right\rfloor$ in Theorem \eqref{F11} as following	
		
		\begin{equation}
			\label{k1} 
			\begin{aligned}
				&\quad \quad \Psi\left( \begin{array}{cc|c} a & b & n \\ \alpha & \beta &\left\lfloor \frac{n}{2} \right\rfloor \end{array} \right)   \\
				&=  \sum_{r=0}^{0} (-1)^{r+\left\lfloor \frac{n}{2} \right\rfloor} \: \frac{\: (n-r-\left\lfloor \frac{n}{2} \right\rfloor-1)! \: \:n \: \:}{(n-2r)! \: r!}
				\left(\begin{array}{c} \lfloor{\frac{n}{2}}\rfloor - r \\ \left\lfloor \frac{n}{2} \right\rfloor \end{array}\right)   \:	\Omega_r\big(\left\lfloor \frac{n}{2} \right\rfloor|\:\alpha, \beta \: | n \big) \:\:
				a^{r} \: (2a-b)^{-r}\\
				&=  (-1)^{\left\lfloor \frac{n}{2} \right\rfloor} \: \frac{\: (n-\left\lfloor \frac{n}{2} \right\rfloor-1)! \: \:n \: \:}{(n)! \: }
				\:	\Omega_0\big(\left\lfloor \frac{n}{2} \right\rfloor|\:\alpha, \beta \: | n \big) \:\:
				\:\\	
				&=  (-1)^{\left\lfloor \frac{n}{2} \right\rfloor} \: \frac{\:	\Omega_0\big(\left\lfloor \frac{n}{2} \right\rfloor|\:\alpha, \beta \: | n \big)  \:}{\:(n-1)(n-2) \cdots  (n-\left\lfloor \frac{n}{2} \right\rfloor) \: }.
			\end{aligned}
		\end{equation}
	From Theorem \eqref{exp1}, it follows  that
		
		\begin{equation}
			\label{k2} 
			\begin{aligned}
				&\quad \quad \Psi\left( \begin{array}{cc|c} a & b & n \\ \alpha & \beta &\left\lfloor \frac{n}{2} \right\rfloor \end{array} \right)   \:= \:  (-1)^{\left\lfloor \frac{n}{2} \right\rfloor} \: \Psi(\alpha, \beta, n).
			\end{aligned}
		\end{equation}
		From \eqref{k1} and \eqref{k2}, we immediately get the proof. 	
		\end{proof}

	Now, from Theorems \eqref{k0} and \eqref{WW3}, we get the following theorem
	
	\begin{theorem}{(The $\Omega-$representation for sums of powers)}
		\label{FA1} 
		\begin{equation}
			\label{FA2} 
			\begin{aligned}
				\frac{x^n+y^n}{(x+y)^{\delta(n)}}\:= \: \frac{\Omega_0\big(\lfloor{\frac{n}{2}}\rfloor |\:xy , - x^2 - y^2 \: |n \big) }{(n-1)(n-2) \cdots (n - \lfloor{\frac{n}{2}}\rfloor )} 	
			\end{aligned}
		\end{equation}
		
	\end{theorem}
	
		Before exploring the second fundamental theorem of Omega sequence, we need to define the Omega space at level $n$ as following.
	\section{The Omega space at level \texorpdfstring{$n$}{}}
	
	\begin{definition}{(The Omega space at level $n$)} \\
		For any given natural number $n$, we define the Omega space, $\omega(n)$, at level $n$ as following  
		\begin{equation}
			\label{Space} 
			\begin{aligned}
			\omega(n):= \{ (a,b)| \: \Psi(a,b,n) \neq 0 \}.
			\end{aligned}
		\end{equation}
	\end{definition}

Here are some specific examples for points  belonging to Omega space at level $n$. 
\begin{equation}
	\label{S1} 
	\begin{aligned}
		\blacksquare \quad	&\text{If}\: \: &n \equiv \pm 1       \pmod{8} \quad  &\Rightarrow  & 	\quad \: \:\Psi(1,0,n) =1 \neq 0 \:  &\Rightarrow & \quad (1,0) &\in \: \omega(n) \\					
		\blacksquare \quad	&\text{If}\: \:  &n \equiv \pm 2       \pmod{12} \quad  &\Rightarrow  & \quad 	\Psi(1,-1,n) = 1 \neq 0 \:  &\Rightarrow & \quad (1,-1) &\in \: \omega(n)\\
		\blacksquare \quad	&\text{If}\: \:  &n \equiv \pm 3       \pmod{16} \quad  &\Rightarrow  & \quad 	\Psi(1,\sqrt{2},n) = -1-\sqrt{2}\neq 0 \:  &\Rightarrow & \quad (1,\sqrt{2}) &\in \: \omega(n)\\
		\blacksquare \quad	&\text{If}\: \:  &n \equiv \pm 4      \pmod{20} \quad  &\Rightarrow  & \quad 	\Psi(1,\varphi -1,n) = -\varphi \neq 0 \:  &\Rightarrow & \quad (1,\varphi -1) &\in \: \omega(n)\\	
		\blacksquare \quad	&\text{If}\: \:  &n \equiv \pm 5       \pmod{24} \quad  &\Rightarrow  & \quad 	\Psi(1,\sqrt{3},n) =  2+ \sqrt{3}\neq 0 \:  &\Rightarrow & \quad (1,\sqrt{3}) &\in \: \omega(n)
	\end{aligned}
\end{equation}
where $\varphi$ is the Golden ratio.

Here are some specific examples for points not belonging to Omega space at level $n$. 
\begin{equation}
	\label{S11} 
	\begin{aligned}
		\blacksquare \quad	&\text{If}\: \: &n \equiv \pm 2       \pmod{8} \quad  &\Rightarrow  & 	\quad \: \:\Psi(1,0,n) =0 \:  &\Rightarrow & \quad (1,0) &\notin \: \omega(n) \\					
		\blacksquare \quad	&\text{If}\: \:  &n \equiv \pm 3       \pmod{12} \quad  &\Rightarrow  & \quad 	\Psi(1,-1,n) =0 \:  &\Rightarrow & \quad (1,-1) &\notin \: \omega(n)\\
		\blacksquare \quad	&\text{If}\: \:  &n \equiv \pm 4       \pmod{16} \quad  &\Rightarrow  & \quad 	\Psi(1,\sqrt{2},n) =0 \:  &\Rightarrow & \quad (1,\sqrt{2}) &\notin \: \omega(n)\\
		\blacksquare \quad	&\text{If}\: \:  &n \equiv \pm 5       \pmod{20} \quad  &\Rightarrow  & \quad 	\Psi(1,\varphi -1,n) =0 \:  &\Rightarrow & \quad (1,\varphi -1) &\notin \: \omega(n)\\	
		\blacksquare \quad	&\text{If}\: \:  &n \equiv \pm 6       \pmod{24} \quad  &\Rightarrow  & \quad 	\Psi(1,\sqrt{3},n) =0 \:  &\Rightarrow & \quad (1,\sqrt{3}) &\notin \: \omega(n)
	\end{aligned}
\end{equation}
 
These examples suggest one to define the Kernel of Omega space for further research developments. 
\begin{definition}{(The Kernel of the Omega space at level $n$)} 
	For any given natural number $n$, we define the Kernel of the Omega space, $Ker_{\omega}(n)$, at level $n$ as following  
	\begin{equation}
		\label{kernel} 
		\begin{aligned}
			Ker_{\omega}(n):= \{ (a,b)| \: \Psi(a,b,n) = 0 \}.
		\end{aligned}
	\end{equation}
\end{definition}

From \eqref{WW4}, we immediately get 
\[       \Psi(xy,-x^2-y^2,n) \neq 0                             \]

for any integers $x,y$ where $x\neq -y $. Hence we
 get 	
\begin{theorem}{ }
	\label{sH} 	For any natural number $n$, the space $\omega(n)$ is infinite and include all the integer points $(xy,-x^2-y^2)$ where $x,y$ any integers such that $x \neq -y$. 
\end{theorem} 

	For any natural number $n > 1$, we should observe from formula \eqref{k00} that  \[ \Psi(\alpha,\beta,n)\neq 0  \quad \iff \quad  \Omega_0\big(\left\lfloor \frac{n}{2}  \right\rfloor|\:\alpha, \beta\: | n \big) \neq 0.\]
	Hence, the following ratio 
	\begin{equation}
		\label{space1} 
		\begin{aligned}
 \: \frac{\:	\Omega_0\big(\left\lfloor \frac{n}{2} \right\rfloor|\:\alpha, \beta\: | n \big)  \:}{ 	\Psi(\alpha, \beta, n) 	     }    		
		\end{aligned}
	\end{equation}	
is well-defined for any point $(\alpha,\beta) \in \omega(n)$. Consequently, for any point $(\alpha,\beta) \in \omega(n)$, the product 
\[      	\:(n-1)(n-2) \cdots  (n-\left\lfloor \frac{n}{2} \right\rfloor) \:    \]
can be represented as following 	
			
		\begin{equation}
			\label{space2} 
			\begin{aligned}
				\:(n-1)(n-2) \cdots  (n-\left\lfloor \frac{n}{2} \right\rfloor) \:    \:= \: \frac{\:	\Omega_0\big(\left\lfloor \frac{n}{2} \right\rfloor|\:\alpha, \beta\: | n \big)  \:}{ 	\Psi(\alpha, \beta, n) 	     }.    		
			\end{aligned}
		\end{equation}	
	
	Substitute $n$ by $2n$ in \eqref{space2}  we get the following theorem
	
	\begin{theorem}{ (The first  fundamental theorem of $\Omega-$sequence) (version 2)}\\
		\label{space3} 
		For any natural number $n$, and any point $(\alpha,\beta) \in \omega(2n)$, we get
	
		\begin{equation}
			\label{space4} 
		\begin{aligned}
			\:n(n+1)  \dotsc (2n-1) \:    \:= \: \frac{\:	\Omega_0\big(n|\:\alpha, \beta\: | 2n \big)  \:}{ 	\Psi(\alpha, \beta, 2n) 	     }.    		
		\end{aligned}
	\end{equation}	
\end{theorem}
	The following unexpected theorem proves that any new prime must be a factor to Omega sequence.
	\begin{notation} 
		Everywhere below, $p_k$ denotes the $k$-th prime, $k >1$. 
	\end{notation}
	\section{The emergence of a new prime number}

	In 1850, Chebyshev proved Bertrand postulate (1845) which states that for any integer $m > 3$ there always exists at least one prime number $p$ with $m < p < 2 m-2$. The postulate is also called the Bertrand–Chebyshev theorem. For a proof of the theorem and for some additional information, for example, see \cite{24}. Therefore, from Bertrand–Chebyshev theorem, one get 	
	\begin{equation}
		\label{space44} 
		\begin{aligned}
		p_{k+1} \: \vert \: 	\:(p_{k}+1)(p_{k}+2)  \dotsc (2{p_k}-3).    		
		\end{aligned}
	\end{equation}

Hence 
	\begin{equation}
	\label{space44e} 
	\begin{aligned}
		p_{k+1} \: \vert \: 	\: p_{k}  (p_{k}+1)(p_{k}+2)  \dotsc (2{p_k}-1).    		
	\end{aligned}
\end{equation}
	
From \eqref{space4}, put $ n= p_{k}$, we get  
\begin{equation}
	\label{space444} 
	\begin{aligned}
		\:p_{k}(p_{k}+1)  \dotsc (2{p_k}-1) \:    \:= \: \frac{\:	\Omega_0\big(p_{k}|\:\alpha, \beta\: | 2p_{k} \big)  \:}{ 	\Psi(\alpha, \beta, 2p_{k}) 	     } 		
	\end{aligned}
\end{equation}	
where $(\alpha,\beta) \in \omega(p_{k})$. Hence from  \eqref{space44e}, and \eqref{space444}, we get the following desirable result

		\begin{theorem}{ (The second fundamental theorem of $\Omega-$sequence)}\\
		\label{space5} 
		For any point $(\alpha,\beta) \in \omega(2p_{k})$, the ratio
		
		\begin{equation}
			\label{space6E} 
			\begin{aligned}
			\frac{\:	\Omega_0\big(p_{k}|\:\alpha, \beta\: | 2p_{k} \big)  \:}{ 	\Psi(\alpha, \beta, 2p_{k}) 	     }
			\end{aligned}
		\end{equation}	
	is integer. And 	
			\begin{equation}
			\label{space6} 
			\begin{aligned}
		p_{k+1}\:  \vert \: \frac{\:	\Omega_0\big(p_{k}|\:\alpha, \beta\: | 2p_{k} \big)  \:}{ 	\Psi(\alpha, \beta, 2p_{k}) 	     }.
			\end{aligned}
		\end{equation}	
	  	Moreover, for any finite set $\varpi \subset \omega(2p_{k}) $, and any finite set $I$ of integers we get 
	  	
	  	\begin{equation}
	  		\label{space8} 
	  		\begin{aligned}
	  					p_{k+1}\:  \: \vert  \: \sum_{\substack{(\alpha,\beta) \in \varpi \\ \lambda \in I}} \: \lambda \: \frac{\:	\Omega_0\big(p_{k}|\:\alpha, \beta\: | 2p_{k} \big)  \:}{ 	\Psi(\alpha, \beta, 2p_{k}) 	     }.    		
	  		\end{aligned}
	  	\end{equation}	
	  \end{theorem} 
We should observe the following generalization for Theorem \eqref{space5}

\begin{theorem}
	\label{gen1} 
	For any point $(\alpha,\beta) \in \omega(2p_{k})$, the ratio 
		
	\begin{equation}
		\label{gen11} 
		\begin{aligned}			
		 \frac{\:	\Omega_0\big(p_{k}|\:\alpha, \beta\: | 2p_{k} \big)  \:}{ 	 p_k(2p_{k}-1)(2p_k-2) \Psi(\alpha, \beta, 2p_{k}) } 			
		\end{aligned}
	\end{equation}
is integer. Moreover 		
	\begin{equation}
		\label{gen111} 
		\begin{aligned}			
			p_{k+1}\:  \vert \: \frac{\:	\Omega_0\big(p_{k}|\:\alpha, \beta\: | 2p_{k} \big)  \:}{ 	 p_k(2p_{k}-1)(2p_k-2) \Psi(\alpha, \beta, 2p_{k}) }.	
		\end{aligned}
	\end{equation}	
\end{theorem} 

We should also observe the following specialization for Theorem \eqref{space5}
	\begin{theorem}
		\label{gen2} 
For any point $(\alpha,\beta)$, we get the following property 
		
		\begin{equation}
			\label{gen22} 
			\begin{aligned}
				p_{k+1}\:  \vert \: \Omega_0\big(p_{k}|\:\alpha, \beta\: | 2p_{k} \big).						
			\end{aligned}
		\end{equation}	
	\end{theorem}

	\section{Some useful special cases }
	
	\subsection{\texorpdfstring{$\Psi(1,1,n)$}{}} For $\alpha=1,\: \beta=1$ we get the following formula for the $\Psi$-sequence
	\begin{equation}
		\label{PP0}
		\Psi(1,1,n) =
		\begin{cases}
			+2	      &   n \equiv \pm 0       \pmod{6} \\
			+1	      &   n \equiv \pm 1       \pmod{6} \\
			-1        &   n \equiv \pm 2       \pmod{6} \\
			-2	      &   n \equiv \pm 3      \pmod{6} \\
		\end{cases}   
	\end{equation}
	Therefore, we get
	
	\begin{theorem}{}
		\label{PP00}
		\begin{equation}
			\label{PP1} 
			\: \frac{\:	\Omega_0\big(\left\lfloor \frac{n}{2} \right\rfloor|\:1, 1\: | n \big)  \:}{\:(n-1)(n-2) \cdots  (n-\left\lfloor \frac{n}{2} \right\rfloor)} \:=
			\begin{cases}
				+2	      &   n \equiv \pm 0       \pmod{6} \\
				+1	      &   n \equiv \pm 1       \pmod{6} \\
				-1        &   n \equiv \pm 2       \pmod{6} \\
				-2	      &   n \equiv \pm 3      \pmod{6} \\
			\end{cases}    
		\end{equation}
	\end{theorem}

\subsection{\texorpdfstring{$\Psi(1,0,n)$}{}} For $\alpha=1,\: \beta=0$ we get the following formula for the $\Psi$-sequence
\begin{equation}
	\label{PP0Q}
	\Psi(1,0,n) \:=
	\begin{cases}
		+2	      &   n \equiv \pm 0       \pmod{8} \\
		+1	      &   n \equiv \pm 1       \pmod{8} \\
		\: 0	      &   n \equiv \pm 2       \pmod{8} \\				
		-1        &   n \equiv \pm 3       \pmod{8} \\
		-2	      &   n \equiv \pm 4      \pmod{8} \\
	\end{cases}     
\end{equation}
Therefore, we get

\begin{theorem}{}
	\label{PP00Q}
	\begin{equation}
		\label{PP1Q} 
		\: \frac{\:	\Omega_0\big(\left\lfloor \frac{n}{2} \right\rfloor|\:1, 0\: | n \big)  \:}{\:(n-1)(n-2) \cdots  (n-\left\lfloor \frac{n}{2} \right\rfloor)} \:=
		\begin{cases}
			+2	      &   n \equiv \pm 0       \pmod{8} \\
			+1	      &   n \equiv \pm 1       \pmod{8} \\
			\: 0	      &   n \equiv \pm 2       \pmod{8} \\				
			-1        &   n \equiv \pm 3       \pmod{8} \\
			-2	      &   n \equiv \pm 4      \pmod{8} \\
		\end{cases}        
	\end{equation}
\end{theorem}

	\subsection{ \texorpdfstring{{$\Psi(1,-1,n)$}}{}}
	 For $a=1,b=-1$ we get the following formula for the $\Psi$-sequence
	\begin{equation}
		\label{peroidicity1}
		\Psi(1,-1,n) =
		\begin{cases}
			+2	      &   n \equiv \pm 0       \pmod{12} \\
			+1	      &  n \equiv \pm 1 , \pm 2      \pmod{12} \\
			\:\:0        &   n \equiv \pm 3       \pmod{12} \\
			-1        &  n \equiv \pm 4, \pm 5       \pmod{12} \\
			-2	      &   n \equiv \pm 6  \pmod{12}
		\end{cases}
	\end{equation}
	Therefore, we get
	
	\begin{theorem}{}
		\label{PP}
		\begin{equation}
			\label{PP1A} 
			\: \frac{\:	\Omega_0\big(\left\lfloor \frac{n}{2} \right\rfloor|\:1, -1\: | n \big)  \:}{\:(n-1)(n-2) \cdots  (n-\left\lfloor \frac{n}{2} \right\rfloor)} \:=
			\begin{cases}
				+2	      &   n \equiv \pm 0       \pmod{12} \\
				+1	      &  n \equiv \pm 1 , \pm 2      \pmod{12} \\
				\:\:0        &   n \equiv \pm 3       \pmod{12} \\
				-1        &  n \equiv \pm 4, \pm 5       \pmod{12} \\
				-2	      &   n \equiv \pm 6  \pmod{12}
			\end{cases}  
		\end{equation}
	\end{theorem}

	\subsection{Combination of Fibonacci and Lucas sequences}
For any natural number $n$, we get, from \eqref{def0}, the following relation
	
	\begin{equation}
		\label{(1,Root5)}
		\Psi(1,\sqrt{5},n) = 
		\begin{cases}
			L(\frac{n}{2}) &   n \equiv 0    \pmod{4} \\
			L(\frac{n+1}{2}) + F(\frac{n-1}{2})\sqrt{5} &  n \equiv 1 \pmod{4} \\
			- F(\frac{n}{2}) \sqrt{5}  &    n \equiv 2 \pmod{4} \\
			-L(\frac{n-1}{2})-F(\frac{n+1}{2})\sqrt{5}& n \equiv 3\pmod{4} 
		\end{cases}.
	\end{equation}
	
	Therefore, from Theorem \eqref{k0}, we immediately get the proof of the following theorem.

	\begin{theorem}{(Representation for a combinations of Fibonacci and Lucas sequences)}
		\label{FL} 
		\begin{equation}
			\label{FL1} 
			\begin{aligned}
			 \:  \frac{\:  \:\Omega_0\big(\left\lfloor \frac{n}{2} \right\rfloor |1,\sqrt{5} |n\big) }{ \: (n-1)(n-2) \cdots (n - \lfloor{\frac{n}{2}}\rfloor )}  = 
			 \begin{cases}
			 	L(\frac{n}{2}) &   n \equiv 0    \pmod{4} \\
			 	L(\frac{n+1}{2}) + F(\frac{n-1}{2})\sqrt{5} &  n \equiv 1 \pmod{4} \\
			 	- F(\frac{n}{2}) \sqrt{5}  &    n \equiv 2 \pmod{4} \\
			 	-L(\frac{n-1}{2})-F(\frac{n+1}{2})\sqrt{5}& n \equiv 3\pmod{4}
			 \end{cases}.	
			\end{aligned}
		\end{equation}
	\end{theorem}

\subsection{The Omega sequence associated with \texorpdfstring{$(1,2)$}{}}
		From \eqref{comp3}
		\[ \Psi(1,2,n)=  \:  \frac{n}{n-\left\lfloor \frac{n}{2} \right\rfloor} \binom{n-\left\lfloor \frac{n}{2} \right\rfloor}{\left\lfloor \frac{n}{2} \right\rfloor} \:(-1)^{\left\lfloor \frac{n}{2} \right\rfloor}.      \]
		
Consequently

		\[ \Psi(1,2,n)= \:(-1)^{\left\lfloor \frac{n}{2} \right\rfloor} \:  2^{\delta(n-1)} \:n^{\delta(n)}.      \]
	Therefore, from Theorem \eqref{k0}, we immediately obtain the following explicit formula	
		
		\begin{theorem}{}
		\label{DA}
		
		\begin{equation}
			\label{AU10} 
			\Omega_0\big(\left\lfloor \frac{n}{2} \right\rfloor|1, 2|n\big) = \:(-1)^{\left\lfloor \frac{n}{2} \right\rfloor} \:  2^{\delta(n-1)} \:n^{\delta(n)}    \:(n-1)(n-2) \cdots  (n-\left\lfloor \frac{n}{2} \right\rfloor) \:.
		\end{equation}
		
	\end{theorem}	
		
	However, if we desire to compute all the terms of Omega sequence that is associated with $n$ and the point $(1,2)$, we need to compute the terms of Omega sequence, one by one, as following:	
	\begin{equation}
		\label{AU1W} 
		\begin{aligned}
			\Omega_r\big(k|1, 2|n\big) &= (- 2)  \: \big(n-2r-\delta(n-1)\big) \: \Omega_{r+1}\big(k-1|1, 2|n\big),  \\
			\Omega_r\big(0|1, 2|n\big) &= 1 \quad   \text{for all} \:\: r.
		\end{aligned}
	\end{equation}
	Solving \eqref{AU1W}, we immediately get

	\begin{theorem}{}
		\label{AU9}
		The Omega sequence associated with $n$ and $(1,2)$ is given by 	
		\begin{equation}
			\label{AU10A} 
			\Omega_r\big(k|1, 2|n\big) = \: (-2)^k \: \prod\limits_{\lambda = 0}^{k-1}\big(n-\delta(n+1) -2r - 2 \lambda \big).
		\end{equation}
		
	\end{theorem}
		
	\subsection{The Omega sequence associated with \texorpdfstring{$(0,-1)$}{}}
	
	Noting that $\Psi(0,-1,n)= 1,$ for any natural number $n$,
	we can deduce the following result
	
	\begin{theorem}{}
		\label{AU11}
		The Omega sequence associated with $n$ and $(0,-1)$ is given by 	
		\begin{equation}
			\label{AU12} 
			\Omega_r\big(k|0, -1|n\big) = \:  \prod\limits_{\lambda = 1}^{k}\big(n-r - \lambda \big).
		\end{equation}
		Moreover 
		\begin{equation}
			\label{AU13} 
			\Omega_0\big(\left\lfloor \frac{n}{2} \right\rfloor |0, -1|n\big) = \:  \:\big(n-1\big) \big(n-2 \big) \cdots  \big(n-\left\lfloor \frac{n}{2} \right\rfloor \big). 
		\end{equation}
	\end{theorem}

\section{New combinatorial identity}	
From \eqref{def0}, it follows that 

\begin{equation}
	\label{ABAB} 	
	\begin{aligned}	
		\Psi(1,-2,n) = 2^{\delta(n+1)}.        
	\end{aligned}	
\end{equation}
	
	 From \eqref{ABAB}, \eqref{AU6NM} and \eqref{k00},  we immediately get the following identity

	\begin{theorem}{(New combinatorial identity)}
		\label{AU7}	
		\begin{equation}
			\label{AU8} 	
			\begin{aligned}	
				\:\big(n-1\big) \big(n-2 \big) \cdots  \big(n-\left\lfloor \frac{n}{2} \right\rfloor \big) 	&=  \\
				\: 2^{ \left\lfloor \frac{n}{2} \right\rfloor - \delta(n+1) }  \: \: \big(n+\delta(n-1) - 2\big)  \: \big(n+\delta(n-1) - 4\big) \: &\big(n+\delta(n-1) - 6\big) \cdots \: \big(3 \big)\: \big(1 \big).
			\end{aligned}	
		\end{equation}
	\end{theorem}
	\subsection*{Remark}

	Put $r=0,\: k = \left\lfloor \frac{n}{2} \right\rfloor$ in \eqref{AU10A}, and compare the result with \eqref{AU10}, we also get identity \eqref{AU8}.
	
		\section{The product of the first odd primes}
Observe that 
		\begin{equation}
			\label{clear} 
			\begin{aligned}
				\prod_{i=2}^{k} p_{i}\:  \vert \: 	\: \big(2p_k-1\big)  \: \big(2p_k-3\big) \: \big(2p_k-5\big) \cdots \: \big(3 \big)\: \big(1 \big).
			\end{aligned}
		\end{equation}	
				Put $n=2p_k$ in \eqref{AU8}, and from \eqref{clear}, we deduce  
		\begin{equation}
		\label{clear2} 
		\begin{aligned}
			\prod_{i=2}^{k} p_{i}\:  \vert \: 	\:\big(2p_k-1\big) \big(2p_k-2 \big) \cdots  \big(2p_k-\left\lfloor \frac{2p_k}{2} \right\rfloor \big).
		\end{aligned}
	\end{equation}				
			From Bertrand–Chebyshev theorem, it follows that 	
				 \begin{equation}
				 	\label{clear3} 
				 	\begin{aligned}
				 		p_{k+1}\:  \vert \: 	\:\big(2p_k-1\big) \big(2p_k-2 \big) \cdots  \big(2p_k-\left\lfloor \frac{2p_k}{2} \right\rfloor \big) .
				 	\end{aligned}
				 \end{equation}	
		From \eqref{clear2}, and \eqref{clear3}, we obtain  
		 \begin{equation}
			\label{clear4} 
		\begin{aligned}
			\prod_{i=2}^{k+1} p_{i}\:  \vert \: 	\:\big(2p_k-1\big) \big(2p_k-2 \big) \cdots  \big(2p_k-\left\lfloor \frac{2p_k}{2} \right\rfloor \big).
		\end{aligned}
		\end{equation}	
		
		Now, from \eqref{space444}, we get the following desirable generalization.
		
				\begin{theorem}{(The third fundamental theorem of $\Omega-$sequence)}\\
			\label{gen5} 
			For any point $(\alpha,\beta) \in \omega(2p_{k})$, the ratio
			\begin{equation}
				\label{gen6} 
				\begin{aligned}
					\frac{\:	\Omega_0\big(p_{k}|\:\alpha, \beta\: | 2p_{k} \big)  \:}{ 	\Psi(\alpha, \beta, 2p_{k}) 	     }
				\end{aligned}
			\end{equation}	
			is integer. And 	
			\begin{equation}
				\label{gen7} 
				\begin{aligned}
					\prod_{i=2}^{k+1} p_{i}\:  \vert \: \frac{\:	\Omega_0\big(p_{k}|\:\alpha, \beta\: | 2p_{k} \big)  \:}{ 	\Psi(\alpha, \beta, 2p_{k}) 	     }.
				\end{aligned}
			\end{equation}	
			Furthermore, for any finite set $\varpi \subset \omega(2p_{k}) $, and any finite set $I$ of integers, we get 
			\begin{equation}
				\label{gen8} 
				\begin{aligned}
						\prod_{i=2}^{k+1} p_{i}\:  \: \vert  \: \sum_{\substack{(\alpha,\beta) \in \varpi \\ \lambda \in I}} \: \lambda \: \frac{\:	\Omega_0\big(p_{k}|\:\alpha, \beta\: | 2p_{k} \big)  \:}{ 	\Psi(\alpha, \beta, 2p_{k}) 	     }.    		
				\end{aligned}
			\end{equation}	
		\end{theorem}
		
	\section{Representation for Chebyshev polynomial sequence}
	
	The Chebyshev polynomials first appeared in his paper \cite{Chebyshev}. 
	The Chebyshev polynomial sequence of the first kind, $T_n(x)$, is defined by 
	\begin{align*}
		T_0(x) & = 1 \\
		T_1(x) & = x \\
		T_{n+1}(x) & = 2 x\:T_n(x) - T_{n-1}(x).
	\end{align*}
	Chebyshev polynomials are important in approximation theory, polynomial approximation, rational approximation, integration, integral equations and in the development of spectral methods for the solution of ordinary and partial differential equations and numerical analysis and in some quadrature rules based on these polynomials such as Gauss-Chebyshev rule that appears in the theory of numerical integration (see for example \cite{10}, \cite{12}). Of all polynomials with leading coefficient unity, it is a well-known property of the Chebyshev polynomials that they possess the smallest absolute upper bound when the argument is allowed to vary between their limits
	of orthogonality and this property suggests the use of Chebyshev polynomials as a means of interpolation. Also, it is well-known that the Nobel Prize-winning physicist Enrico Fermi is the creator of the world's first nuclear reactor, the Chicago Pile-1, and his work led to the discovery of nuclear fission, the basis of nuclear power and the atom bomb,\cite{16} and \cite{17}, and \cite{18}. One of the common approaches to the approximation of Fermi-Dirac integrals is the use of Chebyshev rational approximations, \cite{13},\cite{14},\cite{15}, \cite{19},\cite{20},\cite{21}, to the Fermi-Dirac integrals defined by
	\[F_{s}(x)=\frac{1}{\Gamma\left(s+1\right)}\int_{0}^{\infty}\frac{t^{s}}{e^{t-x}%
		+1}\mathrm{d}t.\]
	Another common approach to the approximation of Riemann Zeta Function is the use of Chebyshev rational approximations, \cite{22}, \cite{23}, where the Riemann Zeta Function, or Euler-Riemann Zeta Function, $\zeta(s)$, is a function of a complex variable $s$ that analytically continues the sum of the Dirichlet series 
	\[ \zeta(s) =\sum_{n=1}^\infty\frac{1}{n^s} \]
	for when the real part of $s$ is greater than 1.
	
	Also, the periodicity of Chebyshev polynomials over finite fields and its tremendous applications on the security of cryptosystems based on Chebyshev polynomials had been studied recently (see for example \cite{6}, \cite{7}, \cite{8}, \cite{9}).
	The integer coefficients of Chebyshev polynomial are given explicitly by the following formula  (see for example \cite{10}, \cite{11}): 
	\begin{equation}
		\label{formula-5}
		T_n(x)=\sum_{i=0}^{\left\lfloor \frac{n}{2} \right\rfloor}(-1)^i \frac{n}{n-i} \binom{n-i}{i} (2)^{n-2i-1} x^{n-2i}.
	\end{equation}
	From \eqref{comp3}, and \eqref{formula-5}, we can deduce the following formula 
	\begin{equation}
		\label{G6-2XCV} 
		\begin{aligned}
			T_n(x) \:= \:  \frac{x^{\delta(n)}}{2^{\delta(n-1)}} \: \Psi(1, 2-4 x^2,n). 	
		\end{aligned}
	\end{equation}
	
	Therefore, from Theorem \eqref{k0}, we immediately get  
	
	\begin{theorem}{(Representation for Chebyshev polynomial sequence)}
		\label{Che} 
		\begin{equation}
			\label{Che2} 
			\begin{aligned}
				T_n(x) \:= \:  \frac{\:x^{\delta(n)}\:\Omega_0\big(\left\lfloor \frac{n}{2} \right\rfloor |1, 2-4 x^2|n\big) }{2^{\delta(n-1)} \: (n-1)(n-2) \cdots (n - \lfloor{\frac{n}{2}}\rfloor )}. 	
			\end{aligned}
		\end{equation}
	\end{theorem}
	
	\subsection{Representation for Dickson polynomial sequence}
	The Dickson polynomial, $D_n(x,\alpha)$, of the first kind of degree $n$ with parameter $\alpha$ is defined by 
	\begin{align*}
		D_0(x,\alpha) & = 2 \\
		D_1(x,\alpha) & = x \\
		D_{n+1}(x,\alpha) & = 2 x\:D_{n}(x,\alpha) - D_{n-1}(x,\alpha).
	\end{align*}
	
	Modern cryptography is heavily based on mathematical theory and computer science practice. Properties of polynomials over finite fields play a vital role not only in mathematics, but also useful in many other applications like error correcting codes, pseudo random sequences used in code-division multiple access (CDMA) systems. CDMA technology was initially used in World War II military operations to thwart enemy attempts to access radio communication signals. As the name suggests, permutation polynomials permute the elements of a ring or field over which they are defined. Permutation Polynomials are the roots of public key methods like RSA Cryptosystem and Dickson cryptographic schemes and they are very important in the development of cryptographic schemes.Recently, permutations of finite fields have become of considerable interest in the construction of cryptographic systems for the secure transmission of data, see \cite{24B}. Also, permutation polynomials have been an active topic of study in recent years due to their important applications in cryptography, coding theory, combinatorial designs theory.  Also the encryption polynomials $x^k$ of the RSA-scheme are replaced by another class of polynomials, namely by the so-called Dickson-polynomials. Cryptographers call this Cryptosystem the Dickson-scheme. Also, Fried \cite{Fried} proved that any integral polynomial that is a permutation polynomial for infinitely many prime fields is a composition of Dickson polynomials and linear polynomials (with rational coefficients). Permutation polynomials were studied first by Hermite \cite{Hermite} and later by Dickson \cite{Dickson} and \cite{Dickson2}. Dickson polynomials form an important class of permutation polynomials and have been extensively investigated in recent years under different contexts. See for instance \cite{L-1},\cite{L-2},    \cite{L-3},\cite{L-4},\cite{L-5}, \cite{L-6}, \cite{L-7}, and \cite{L-8}, where the work on Dickson polynomials, and its developments are presented.  
	
	For integer $n> 0$ and $\alpha$ in a commutative ring $R$ with identity. The Dickson polynomials (of the first kind) over R are given by 
	
	\begin{equation}
		\label{formula-2}
		D_n(x,\alpha)=\sum_{i=0}^{\left\lfloor \frac{n}{2} \right\rfloor}\frac{n}{n-i} \binom{n-i}{i} (-\alpha)^i x^{n-2i}
	\end{equation}
	From \eqref{comp3},and \eqref{formula-2} we can deduce the following formula 
	\begin{equation}
		\label{G6-2XCVW} 
		\begin{aligned}
			D_n(x,\alpha) \:= \:  x^{\delta(n)} \: \Psi(\alpha, 2\alpha-x^2,n). 	
		\end{aligned}
	\end{equation}
	Therefore, from Theorem \eqref{k0}, we immediately get the proof of the following theorem.

	\begin{theorem}{(Representation for Dickson polynomial sequence)}
		\label{Dic} 
		\begin{equation}
			\label{G6-2XCW} 
			\begin{aligned}
				D_n(x,\alpha)  \:= \:  \frac{\: x^{\delta(n)} \:\Omega_0\big(\left\lfloor \frac{n}{2} \right\rfloor |\alpha, 2\alpha-x^2|n\big) \: }{ \: (n-1)(n-2) \cdots (n - \lfloor{\frac{n}{2}}\rfloor )}. 	
			\end{aligned}
		\end{equation}
	\end{theorem}
		
	\section{Mersenne primes and even perfect numbers}
Mersenne numbers $2^{p} - 1$ with prime $p$ form the sequence \[3, 7, 31,
127, 2047, 8191, 131071, 524287, 8388607, 536870911, \dotsc \] (sequence
\seqnum{A001348} in \cite{Slo}). 
For $2^{p} - 1$ to be prime, it is necessary that $p$ itself be prime.	
		
	\subsection{Primality test for Mersenne primes}
	\begin{theorem}{(Lucas-Lehmer-Moustafa)}
		\label{U14}
		Given prime $p \geq 5$. The number $2^p-1$ is prime if and only if 	
		\begin{equation}
			\label{U15} 
				2n-1 \quad  \vert \quad \Psi(1,4,n),
		\end{equation}
		where $n:=2^{p-1}$.
	\end{theorem}
	
	\begin{proof}
		Given prime $p \geq 5$, let $n:=2^{p-1}$. From Lucas-Lehmer test, \cite{Jean}, we have 
		\begin{equation*}
			\begin{aligned}
				2^p -1 \quad \text{is prime} \quad  \iff 2^p -1  \quad | \quad (1+\sqrt{3})^n + (1-\sqrt{3})^n.	
			\end{aligned}
		\end{equation*}
		As $n$ even, $\delta(n)=0$, and from Theorem \eqref{WW3}, we get the following equivalent statement:
		\begin{equation*}
			\begin{aligned}
				2^p -1 \quad \text{is prime} \quad  \iff 2^p -1  \quad | \quad 
				\Psi(x_0 \: y_0,-x_0^2-y_0^2, n),
			\end{aligned}
		\end{equation*}
		where $x_0= 1+\sqrt{3}, \quad  y_0= 1-\sqrt{3}$. As $(x_0 \: y_0,-x_0^2-y_0^2) = (-2, -8),$ and from Theorem \eqref{W11}, and noting $(2^p -1, 2 ) =1 $,  we get the following equivalent statements:
		
		\begin{equation*}
			\begin{aligned}
				2^p -1 \quad \text{is prime} \quad  &\iff 2^p -1  \quad | \quad \Psi(-2,-8,n) \\ 
				&\iff 2^p -1  \quad | \quad(-2)^{\lfloor{\frac{n}{2}}\rfloor}\: \Psi(1,4,n) \\ 	
				&\iff 2^p -1  \quad | \quad  \Psi(1,4,n).  
			\end{aligned}
		\end{equation*}	
	\end{proof}
	From Theorem\eqref{U14} and Theorem\eqref{k0}, we immediately get the following result 
	
	\begin{theorem}{(Lucas-Lehmer-Moustafa)}
		\label{U16}
		Given prime $p \geq 5$. The number $2^p-1$ is prime if and only if 	
		\begin{equation}
			\label{U17} 
				2n-1 \quad  \vert \quad  \: \frac{\:	\Omega_0\big(\left\lfloor \frac{n}{2} \right\rfloor|\:1, 4\: | n \big)  \:}{\:(n-1)(n-2) \cdots  (n-\left\lfloor \frac{n}{2} \right\rfloor) \: },
		\end{equation}
		where $n:=2^{p-1}$.
	\end{theorem}
	Therefore, from Theorem \eqref{U16} and from Euclid-Euler theorem for even perfect numbers, we get the following result
	\begin{theorem}{(Euclid-Euler-Lucas-Lehmer-Moustafa)}
		\label{U18} 
		A number $N$ is even perfect number if and only if $N=2^{p-1}(2^p-1)$ for some prime $p$, and 
		\begin{equation}
			\label{U19} 
				2n-1 \quad  \vert \quad  \: \frac{\:	\Omega_0\big(\left\lfloor \frac{n}{2} \right\rfloor|\:1, 4\: | n \big)  \:}{\:(n-1)(n-2) \cdots  (n-\left\lfloor \frac{n}{2} \right\rfloor) \: },
		\end{equation}
		where $n:=2^{p-1}$.	
	\end{theorem}
	From Theorem\eqref{k0}, and the fact that 
	
	\begin{equation}
		\label{WQ11}  \Psi(-2,-5,p) = 2^p -1, 
	\end{equation}	
	
	we immediately get the following desirable representation for Mersenne numbers.
	
	\begin{theorem}{(Representation of Mersenne numbers)}
		\label{Theorem of G2f}
		For any given odd natural number $p$, the number $2^p -1$ can be represented by
		\begin{equation}
			\label{G2Bf} 
			\begin{aligned}
				2^p-1 \:=\: \frac{	\Omega_0\big(\left\lfloor \frac{p}{2} \right\rfloor|-2, -5\: | p \big)}{(p-1)(p-2) \cdots (p - \lfloor{\frac{p}{2}}\rfloor )}, 	
			\end{aligned}
		\end{equation}
		
		where the double-indexed polynomial sequence $\Omega_r(k)$ is associated with the point $(-2,-5)$, $0 \leq r + k \leq  \lfloor{\frac{n}{2}}\rfloor $, and defined by the recurrence relation 
		\begin{equation}
			\label{G2-e} 
			\begin{aligned}
				\Omega_r(k) &=  \: (p-r-k) \: \Omega_r(k-1) + 4 \: (p-2r) \: \Omega_{r+1}(k-1),  \\
				\Omega_r(0) &= 1   \quad \text{for all} \quad r.
			\end{aligned}
		\end{equation}
	\end{theorem}
	
	Hence from Theorem \eqref{U16} and Theorem \eqref{Theorem of G2f}, we get the following

	\begin{theorem}{(Lucas-Lehmer-Moustafa)} 
		\label{Theorem of ABCD12}
		For any given prime $p\geq 5$, $n:=2^{p-1}$, the number $2^p -1$ is prime if and only if 
		\begin{equation}
			\label{ABCD4} 
			\begin{aligned}
					\frac{\:	\Omega_0\big(\left\lfloor \frac{p}{2} \right\rfloor|-2, -5\: | p \big)  \:}{(p-1)(p-2) \cdots (p - \lfloor{\frac{p}{2}}\rfloor )}  \quad  \vert \quad \frac{\:	\Omega_0\big(\left\lfloor \frac{n}{2} \right\rfloor|\:1, 4\: | n \big)  \:}{(n-1)(n-2) \cdots (n - \lfloor{\frac{n}{2}}\rfloor )}.
			\end{aligned}
		\end{equation}
		
	\end{theorem}
		
Now, from Theorem \eqref{AU11}, we get		
	
			\begin{equation}
			\label{AU13G}
			\begin{aligned} 
		 \:\big(n-1\big) \big(n-2 \big) \cdots  \big(n-\left\lfloor \frac{n}{2} \right\rfloor \big) &= 	\Omega_0\big(\left\lfloor \frac{n}{2} \right\rfloor |0, -1|n\big),\\ 
		 \:\big(p-1\big) \big(p-2 \big) \cdots  \big(p-\left\lfloor \frac{p}{2} \right\rfloor \big) &= 	\Omega_0\big(\left\lfloor \frac{p}{2} \right\rfloor |0, -1|p\big).
		 \end{aligned}
		\end{equation}
Consequently, from 	Theorem \eqref{Theorem of ABCD12}, and from \eqref{AU13G}, we get 	
	\begin{theorem}{(Lucas-Lehmer-Moustafa)} 
		\label{Theorem of ABCD12G}
		For any given prime $p\geq 5$, $n:=2^{p-1}$. The number $2^p -1$ is prime if and only if 
		\begin{equation}
			\label{ABCD4G} 
			\begin{aligned}
			\Omega_0\big(\left\lfloor \frac{n}{2} \right\rfloor|\:0, -1\: | n \big) \:	\Omega_0\big(\left\lfloor \frac{p}{2} \right\rfloor|-2, -5\: | p \big)   \quad  \vert \quad \Omega_0\big(\left\lfloor \frac{n}{2} \right\rfloor|\:1, 4\: | n \big) \: \Omega_0\big(\left\lfloor \frac{p}{2} \right\rfloor|0, -1\: | p \big)  .
			\end{aligned}
		\end{equation}
		
	\end{theorem}

	\subsection{ Mersenne composite numbers	}
	The number $2^p-1$ is called Mersenne composite number if $p$ is prime but  $2^p-1$ is not prime.
	\begin{theorem}{}
		\label{Maresenn composite}
		Given prime $p$, $n:=2^{p-1}$. If 
		\begin{equation}
			\label{composite} 
				2n-1 \quad  \vert \quad \Psi(1,4,n\: \pm 1),
		\end{equation}
		then $2^p-1$ is a Mersenne composite number.
	\end{theorem}
	\begin{proof}
		Let $2^p-1   \vert  \Psi(1,4,n\: \pm 1)  $. Now, suppose the contrary, and let $2^p-1$ be prime. Then from Theorem \eqref{U14},  and from the recurrence relation \eqref{def0}, we get $2^p-1  \: \vert \:  \Psi(1,4,1) =1$. Contradiction.
	\end{proof}

	\section{Clarifications on the theorems of the summary}
	
	Now need to provide clarifications for how the theorems in the summary arose up easily, one by one, as special cases of the main theorems of the this paper. 
	\begin{itemize}
		\item Clarifications on Theorem \eqref{Theorem of ABC1}:\\
It is clear that the sequences $ A_r(k) $ and  $ B_r(k) $, which are defined in \eqref{ABC2}, satisfy:  
	\begin{equation}
	\label{ABC2D2} 
	\begin{aligned}		
		A_r(k) &= \Omega_r\big(k|-2, -5\: | p \big),\\
	B_r(k) &= 	\Omega_r\big(k|\:1, 4\: | n \big).
	\end{aligned}
\end{equation}
Hence from Theorem \eqref{Theorem of ABCD12} and Theorem \eqref{k0} we get the proof of Theorem \eqref{Theorem of ABC1}.

\item Clarifications on Theorem \eqref{Theorem of G2fQ}: \\
	Again, it is clear that 
	\begin{equation}
		\label{ABC2D3} 
		\begin{aligned}		
			A_r(k) &= \Omega_r\big(k|-2, -5\: | p \big).
		\end{aligned}
	\end{equation}
	Hence from Theorem \eqref{Theorem of G2f} and Theorem \eqref{k0} we get the proof of Theorem \eqref{Theorem of G2fQ}.
	
\item Clarifications on Theorem \eqref{KG11}: \\
	Again, it is clear that 
	\begin{equation}
		\label{ABC2D3D} 
		\begin{aligned}		
			U_r(k) &= \Omega_r\big(k|1, 1\: | n \big).
		\end{aligned}
	\end{equation}
	Hence from Theorem \eqref{PP00} we get the proof of Theorem \eqref{KG11}.
	
\item Clarifications on Theorem \eqref{KG11Q}: \\
We should notice that  
\begin{equation}
	\label{ABC2D4D} 
	\begin{aligned}		
		V_r(k) &= \Omega_r\big(k|1, 0\: | n \big).
	\end{aligned}
\end{equation}
Hence from Theorem \eqref{PP00Q} we get the proof of Theorem \eqref{KG11Q}.

	\item Clarifications for Theorem \eqref{PtP}: \\
	We should notice that  
	\begin{equation}
		\label{ABC2D5D} 
		\begin{aligned}		
			W_r(k) &= \Omega_r\big(k|1, -1\: | n \big).
		\end{aligned}
	\end{equation}
	Hence from Theorem \eqref{PP} we get the proof of Theorem \eqref{PtP}.
	
	\item Clarifications on Theorem \eqref{Theorem of G5}: \\
	We should notice that  
	\begin{equation}
		\label{ABC2D6D} 
		\begin{aligned}		
			T_r(k) &= \Omega_r\big(k|1, -2\: | n \big).
		\end{aligned}
	\end{equation}
	Hence from \eqref{ABAB} and Theorem \eqref{k0} we get the proof of Theorem \eqref{Theorem of G5}.
	
	\item Clarifications on Theorem \eqref{Theorem of G6}:\\
	We should notice that  
	\begin{equation}
		\label{ABC2D7D} 
		\begin{aligned}		
			H_r(k) &= \Omega_r\big(k|-1, -3\: | n \big).
		\end{aligned}
	\end{equation}
It is straightforward to see \[\Psi(-1,-3,n)= L(n).\]
 Hence from Theorem \eqref{k0} we get the proof of Theorem \eqref{Theorem of G6}.
	
	\item	Clarifications on Theorem \eqref{Theorem of G4}:\\
	We should notice that  
	\begin{equation}
		\label{ABC2D8D} 
		\begin{aligned}		
			F_r(k) &= \Omega_r\big(k|-2, -5\: | 2^n \big).
		\end{aligned}
	\end{equation}
	Noting that \[\Psi(-2,-5, 2^n)= F_n,\]
	 hence, from Theorem \eqref{k0}, we get the proof of Theorem \eqref{Theorem of G4}.

	\item	Clarifications on Theorem \eqref{Theorem of G7}:\\
	We should notice that  
	\begin{equation}
		\label{ABC2D9D} 
		\begin{aligned}		
			G_r(k) &= \Omega_r\big(k|1, -3\: | n \big).
		\end{aligned}
	\end{equation}
	Noting that \[\Psi(1,-3, n)= \begin{cases}
		F(n)	      &   \text{if} \quad n \:\: odd \\
		L(n)	      &  \text{if} \quad n \: \: even  \\
	\end{cases},\]
	hence, from Theorem \eqref{k0}, we get the proof of Theorem \eqref{Theorem of G7}.

\subsection{Fibonacci-Lucas oscillating sequence}
	It is nice to study the sequence $G(n)$
		\[G(n)= \begin{cases}
		F(n)	      &   \text{if} \quad n \:\: odd \\
		L(n)	      &  \text{if} \quad n \: \: even  \\
	\end{cases}\]
		and I call it Fibonacci-Lucas oscillating sequence.  Theorem \eqref{Theorem of G7} should motivate researchers for further new investigations towards the arithmetic of this  sequence. The sequence $G(n)$ corresponds to the sequence \seqnum{A005247} in \cite{Slo}. This sequence alternates Lucas \seqnum{A000032} and Fibonacci \seqnum{A000045} sequences for even and odd $n$. 

	\end{itemize}

	\section{Further research investigations}
	\subsection{Fibonacci numbers and \texorpdfstring{{$p_k$}}{}}
	With some extra work one can prove the following result:
	
		For any given natural number $n$, if we associate the double-indexed polynomial sequence $\Lambda_r(k)$ which is defined by
		\begin{equation}
			\label{G6X} 
			\begin{aligned}
				\Lambda_r(k) &=   \: (n-r-k) \: \Lambda_r(k-1) + 2\: (n-1-2r - \: \delta(n)) \: \Lambda_{r+1}(k-1), \\
			\Lambda_r(0) &= 1   \quad   \text{for all} \:\: r,
			\end{aligned}
		\end{equation}
		then 
		\begin{equation}
			\label{G6-2X} 
			\begin{aligned}
				F(n) \:= \: \frac{\Lambda_0(\lfloor{\frac{n-1}{2}}\rfloor   )}{(n-1)(n-2) \cdots (n - \lfloor{\frac{n-1}{2}}\rfloor )}, 	
			\end{aligned}
		\end{equation}
		
		where $F(n)$ is Fibonacci sequence. Moreover 
	\begin{equation}
		\label{primeFib} 
	\begin{aligned}
		p_{k+1}\:  \vert \: \: \frac{\Lambda_0(p_{k}-1) }{F(2p_{k})}.    		
	\end{aligned}
\end{equation}

		\subsection{Conjecture}
		While my studies for the connections between Omega sequences, Omega spaces, the natural relationships between Omega sequences and the prime $p_{k+1}$, and my studies how the primes naturally arise up, I feel strongly compelled to propose the following conjecture:
	\begin{conjecture}{(The Omega conjecture for the prime numbers)}
		\label{conj1}
	Based on the properties of  Omega sequences, Omega spaces, and the Kernel of Omega space, we can find algorithm to determine the prime $p_{k+1}$ based only on the knowledge of $p_{k}, p_{k-1},p_{k-2}, \dotsc $ and on the knowledge of some points in Omega space at level $n$ for some natural number $n$, where $n$ is independent on $p_{k+1}$.
\end{conjecture}	
	
	Of course, my remark only begins the story, and I have told only of those formulas that show various natural deep connections of Omega sequences with the nature of prime numbers, Mersenne primes, even perfect numbers, and unify many well-known sequences in number theory.

	\subsection*{Acknowledgments}
	
	I would like to deeply thank University of Bahrain for their support. Also special great thanks to the eminent professors Bruce Reznick, Bruce Berndt, and Alexandru Buium for the enthusiasm and support  they gave me for many years. Also special thanks to University of Illinois at Urbana-Champaign, The Graduate Research Center of the City University of New York, CUNY, Pacific Institute for Mathematical Sciences,  Vancouver, Canada, University of California at Los Angeles, UCLA, Lorentz Institute, Louisiana State University, University of Warwick, University of Turku of Finland, San Francisco State University, and Isaac Newton Institute for various grants and support that greatly inspired me.

	\medskip
\end{document}